\documentclass[reqno,12pt]{amsart}

\usepackage{amsmath,amsthm,verbatim,amssymb,amsfonts,amscd,graphicx}
\usepackage{xcolor,slashed,kotex}
\usepackage{soul}
\usepackage{graphics}
\usepackage{array}
\usepackage{hhline}
\usepackage{enumitem}

\usepackage{euscript}
\usepackage{url,hyperref}
\usepackage{mathrsfs}

\newcommand{\p}{\partial}

\newcommand{\la}{\langle}
\newcommand{\ra}{\rangle}
\newcommand{\na}{\nabla}

\newcommand{\eps}{\varepsilon}

\newcommand{\be}{\begin{equation}}
\newcommand{\ba}{\begin{aligned}}
\newcommand{\bee}{\begin{equation*}}
\newcommand{\ee}{\end{equation}}
\newcommand{\ea}{\end{aligned}}
\newcommand{\eee}{\end{equation*}}
\newcommand{\bea}{\begin{equation} \begin{aligned} }
\newcommand{\eea}{\end{aligned}\end{equation} }

\newcommand{\R}{{ \mathbb{R}  }}

\newcommand{\Tn}{{ \mathbb{T}^n  }}

\newcommand{\om}{{ \omega   }}

\newcommand{\nn}{{ \nonumber }}

\newcommand{\cE}{{\mathcal{E}}}
\newcommand{\cF}{{\mathcal{F}}}

\newcommand{\cH}{ {\mathcal{H}_{\mu_*}} }

\newcommand{\cL}{{\mathscr{L}}}
\newcommand{\cM}{\mathcal{M}}
\newcommand{\cN}{{\mathcal{N}}}

\newcommand{\cP}{{\mathcal{P}}}

\newcommand{\cU}{{ \mathcal{U}   }}
\newcommand{\cV}{{ \mathcal{V}   }}
\newcommand{\cW}{{\mathcal{W}}}

\newcommand{\vphi}{\varphi}

\newcommand{\si}{\sigma}

\newcommand{\T}{\tau}

\newcommand{\cK}{\mathcal{K}}

\newcommand{\z}{\zeta}
\newcommand{\tsi}{\tilde{\sigma}}

\newcommand{\tcU}{\widetilde{\mathcal{U}}}
\newcommand{\tcV}{\widetilde{\mathcal{V}}}
\newcommand{\tcW}{\widetilde{\mathcal{W}}}

\newcommand{\rw}{\rightarrow}

\newcommand{\al}{\alpha}

\newcommand{\ep}{\varepsilon}

\newcommand{\Tr}{\textrm{tr}}

\newcommand{\Div}{\operatorname{div}}

\newtheorem{theorem}{Theorem}[section]

\newtheorem{lemma}[theorem]{Lemma}
\newtheorem{corollary}[theorem]{Corollary}

\newtheorem{example}[theorem]{Example}

\theoremstyle{remark}

\newtheorem{remark}[theorem]{Remark}

\theoremstyle{definition}

\newtheorem{definition}[theorem]{Definition}

\numberwithin{equation}{section}

\title[Convergence of Wasserstein gradient flow]{Wasserstein–{\L}ojasiewicz inequalities and asymptotics of McKean–Vlasov equation}

\author{Beomjun Choi}
\address{BC: Department of Mathematical Sciences, Korea Advanced Institute of Science and Technology, Daejeon 34141, Republic of Korea}
\email{bchoi@kaist.ac.kr}

\author{Seunghoon Jeong}
\address{SJ: Department of Mathematics, Pohang University of Science and Technology, Pohang, 37673, Republic of Korea}
\email{jshmth@postech.ac.kr}

\author{Geuntaek Seo}
\address{GS: Department of Mathematics, Pohang University of Science and Technology, Pohang, 37673, Republic of Korea}
\email{gtseo@postech.ac.kr}

\begin{document}

\begin{abstract} 
We prove convergence to equilibrium for solutions to the McKean–Vlasov (granular media) equation on the flat torus in a genuinely nonconvex setting. Our approach is based on a Wasserstein–Łojasiewicz gradient inequality for the associated free energy, established under mild analyticity assumptions on the confinement and interaction potentials.  This yields convergence of the corresponding Wasserstein gradient flow without convexity assumption and without postulating log-Sobolev related functional inequalities. We expect this strategy to extend to more general nonconvex Wasserstein gradient flows. In the present work we develop it in the McKean–Vlasov setting, with the Keller–Segel chemotaxis model on the torus as a prominent application.
\end{abstract}

\maketitle

\section{Introduction}

In this paper, we study the asymptotic convergence of solutions to the McKean-Vlasov equation (also known as the granular media equation)
\begin{equation}\label{MV}
\partial_t \rho
=  \Delta \rho
+ \nabla \cdot \big(\rho \,\nabla (V + W * \rho)\big).
\end{equation}
This nonlocal, nonlinear Fokker-Planck equation describes the evolution of a density of randomly moving particles under the combined effects of an external confinement potential $V$ and pairwise interactions modeled by $W$. The equation admits the free energy functional
\begin{equation}\label{energy_functional}
\cF(\mu):=   \int   \log \rho \,d\mu  +  \int  V \, d\mu + \frac{1}{2} \iint  W(x-y) \, d\mu(x)d\mu(y),
\end{equation}
where $\mu=\rho\,dx$, and \eqref{MV} can be formulated as the $2$-Wasserstein gradient flow of $\mathcal{F}$ on the space of absolutely continuous probability measures $\mathcal{P}_2^{ac}$, within the formal Riemannian calculus pioneered by Otto \cite{MR1842429}.

At the microscopic level, \eqref{MV} arises as the mean-field limit of an interacting $N$-particle Langevin dynamics driven by the same confinement and interaction potentials $V$ and $W$ (see \eqref{N_SDE}). Such systems and their continuum limit equation \eqref{MV} appear in statistical physics, interacting particle systems, plasma models, and, more recently, in the training of neural networks in machine learning and in sampling algorithms. In all these contexts, understanding the long-time behavior and convergence to equilibrium of the dynamics is of central importance.

The long-time behavior of \eqref{MV} has been extensively studied. 
Through pioneering works, convergence to equilibrium is well understood under displacement convexity of the free energy, which holds when $V$ and/or $W$ are convex. 
Alternatively, convergence follows from functional inequalities such as log-Sobolev or Talagrand inequalities, which effectively impose a version of strict convexity of the energy landscape around the equilibrium; see Remark~\ref{remark-LSLS}. 
However, these conditions exclude many nonconvex settings of practical relevance, and the general case has remained largely open.

\medskip


\noindent \textbf{Main result.}
We establish convergence to equilibrium for \eqref{MV} \emph{without imposing any convexity (or small-nonconvexity) assumptions on $V$ and $W$, and without assuming \emph{a priori} functional inequalities for $\mathcal{F}$ that are typically derived from such assumptions}; see Theorem~\ref{thm-mainconvergence}. For clarity, we work throughout on the flat torus $\mathbb{R}^n/\mathbb{Z}^n$, $n\ge1$.

The key ingredient, proved in Theorem~\ref{thm-main-LS}, is a {\L}ojasiewicz inequality for functionals on the Wasserstein space. 
Such an inequality implies a coercivity estimate for the entropy dissipation near stationary solutions to gradient flow. 
Our main contribution is a methodology for establishing Wasserstein--{\L}ojasiewicz inequalities for a broad class of analytic energies on the Wasserstein space. We develop this approach using \eqref{MV} as a guiding model problem; in this setting, the resulting coercivity yields convergence of every gradient-flow trajectory of \eqref{MV} to a stationary solution, provided no blow-up occurs.

Our main structural assumption is analyticity of $V$ and $W$, allowing $W$ to have a Coulomb-type singularity; see Assumptions~\textbf{(\ref{assump_A1})}--\textbf{(\ref{assump_A3})}. As a key example within this framework, we show that the parabolic--elliptic Keller--Segel chemotaxis system on the torus converges to equilibrium under the same boundedness assumption; see Theorem~\ref{cor-KSeq}. Blow-up may occur when the singular attraction encoded by $W$ dominates diffusion. By contrast, when $W$ is regular, Theorem~\ref{thm-noblowup} shows that solutions remain uniformly bounded for all time.
\medskip

\noindent\textbf{{\L}ojasiewicz inequality.}
For a Euclidean gradient flow $\dot{x} = -\nabla f(x)$, a {\L}ojasiewicz inequality around a critical point $x_0$ takes the form
\begin{equation}\label{167}
|\nabla f(x)| \ge c\, |f(x)-f(x_0)|^\theta,
\qquad \theta\in[1/2,1),
\end{equation}
and holds for every analytic $f$ \cite{lojasiewicz1965ensembles,Loj3}. 
Combined with the dissipation identity $\frac{d}{dt}f(x(t))=-|\nabla f(x(t))|^2$, it implies convergence of any trajectory with an $\omega$-limit\footnote{For an exposition of this {\L}ojasiewicz argument, we refer to \cite[Section~1]{KMP}. We also follow this approach in the proof of Theorem~\ref{thm-mainconvergence}.}, with rates depending on~$\theta$; see Corollary~\ref{corollary-rate}. 
The case $\theta=1/2$ corresponds to the classical Polyak-{\L}ojasiewicz inequality and yields exponential convergence. 

In the Wasserstein setting, establishing an analogue of \eqref{167} requires nontrivial work, both because the underlying Wasserstein space lacks linear structure and because we work in an infinite-dimensional setting. Nevertheless, we show that an appropriate Wasserstein-{\L}ojasiewicz inequality holds for~$\mathcal{F}$, and this is the key ingredient in the convergence result described above.

\medskip

We now fix notation for the three components of the energy \(\mathcal F\):
\[
\mathcal F = \mathcal U + \mathcal V + \mathcal W,
\]
with \(\mathcal U\), \(\mathcal V\), \(\mathcal W\) denoting the internal, confinement, and interaction energies, respectively. We will also identify \(\mu\) with its density \(\rho\) when no confusion arises.

As is now well-known, starting from \cite{JKO,MR1842429}, equation \eqref{MV} fits into the general Wasserstein gradient flow framework and encompasses many important evolutionary equations, including the heat equation, porous medium equation, Fokker-Planck equation, and aggregation-diffusion equations (e.g.\ the Keller-Segel system). All of these can be written in the compact form
\[ 
\p_t \rho = \nabla \cdot \Big  (\rho \nabla  \frac{\delta \mathcal{F}}{\delta \rho } \Big  ),
\]
where $\frac{\delta\mathcal F}{\delta\rho}$ denotes the first variation of $\rho\mapsto\mathcal F(\rho dx)$ in the $L^2$ sense (also known as the variational derivative).

\medskip 
\medskip 

\noindent \textbf{Prior results and background.} 
The long-time behavior of solutions to \eqref{MV} has been extensively studied, both at the PDE level and at the level of the associated interacting particle systems. We briefly review several main directions, in order to contrast them with our nonconvex, non-functional-inequality framework.

\smallskip
\noindent\emph{Convex gradient-flow setting.}
A significant body of work has investigated the stability and long-time behavior of McKean-Vlasov dynamics under convexity assumptions on the energy. A seminal work in this direction is the work of Carrillo, McCann and Villani \cite{MR2209130, MR2053570}, who showed that the asymptotic convergence of Wasserstein gradient flows on $\cP_2(\R^n)$ occurs at either exponential or algebraic rates, depending on the convexity properties of the entropy density $U$, the confinement potential $V$, and the interaction potential $W$, including certain degenerate cases. 
In \cite{MR3035983}, the authors improved upon the results of \cite{MR2209130, MR2053570} under the assumptions of convex potentials $V$ and $W$, imposing positive $\lambda$-convexity only for $V$ on $|x|\geq R$.

\smallskip
\noindent\emph{Nonconvex regimes and functional inequalities.}
In genuinely nonconvex settings, and in the absence (until now) of a {\L}ojasiewicz-type inequality in Wasserstein space, the asymptotic convergence of solutions to~\eqref{MV} has often been studied under \emph{functional inequalities}, such as LSI, Talagrand transportation inequalities, or Bakry-Émery curvature lower bounds.
For instance, \cite[Theorem~4.3]{MR3035983} establishes asymptotic convergence even when the potentials \(V\) and \(W\) are not convex, by assuming a Talagrand-type inequality. 
Similarly, \cite{MR2357669} derives convergence based on a transportation cost-information inequality (the so-called \(T_1\) Talagrand inequality), without requiring uniform convexity of either the confinement or interaction potentials. Moreover, in the $\lambda$-geodesically convex setting, \cite{hauer2019kurdyka} demonstrates an equivalence between Talagrand inequalities, LSI, and a {\L}ojasiewicz-type inequality.

Global LSI generally fails in the presence of multiple equilibria. Instead, \cite{Monmarche:2025aa} employs a \emph{local} LSI to prove local exponential stability of non-degenerate local minimizers, which includes the quadratic attractive interaction in a double-well potential (the Curie--Weiss model) at sub-critical temperatures.\footnote{In the literature, the internal energy is often scaled by a factor $\beta^{-1}$ (representing temperature). Under this convention, the diffusion term in \eqref{MV} becomes $\beta^{-1}\Delta\rho$, facilitating the comparison of the relative strengths of diffusion and aggregation. By rescaling $V$, $W$, and time, one can normalize $\beta$ to $1$. Since our assumptions on $V$ and $W$ in \textbf{(\ref{assump_A1})}--\textbf{(\ref{assump_A3})} remain invariant under such rescaling, we set $\beta=1$ in \eqref{MV} without loss of generality.} When the temperature is critical, \cite{Monmarche:2025aa} shows an algebraic rate of convergence to the unique minimizer. (See also \cite[Section 6]{monmarche2025free} for a sharp rate derived via \emph{degenerate} LSI\footnote{This can be understood as a derivation of the Wasserstein-\L ojasiewicz inequality in Theorem \ref{thm-main-LS} for this example by a different method.}.) Let us also mention the paper \cite{BV25}, which analyzes the Wasserstein gradient flow of the Coulomb/Riesz discrepancy energy and establishes a Polyak-{\L}ojasiewicz inequality along the flow, yielding exponential convergence to the target measure. Finally, \cite{eberle2019quantitative} proved exponential convergence via probabilistic methods, assuming the confinement potential is convex (dissipative) at infinity and the interaction is a small perturbation.

\smallskip
\noindent\emph{Aggregation-diffusion and Keller-Segel models.}
As a prominent subclass of \eqref{MV}, the aggregation-diffusion equation corresponding to $V = 0$ has been extensively investigated in various contexts (see, for instance, \cite{carrillo2018aggregationdiffusionequationsdynamicsasymptotics, MR4022083, MR3842067, MR3165280} and the references therein). 
A notable result established in \cite{MR4022083} shows that, for purely attractive interaction kernels, all stationary solutions in $L_+^1 \cap L^\infty(\mathbb{R}^n)$ are radially symmetric and strictly decreasing. 
Furthermore, in the two-dimensional Newtonian case, the equilibrium is unique up to spatial translation, and any solution converges to this profile as $t \to \infty$. The question of uniqueness versus nonuniqueness of equilibria in the presence of porous-medium type diffusion was addressed in \cite{MR4373166}.

 The parabolic-elliptic Keller-Segel system is obtained by taking $V=0$ and $W=\mathcal{N}$ in \eqref{MV}, where $\mathcal{N}$ denotes the Newtonian potential. 
In this case, the competition between attractive aggregation and diffusive spreading leads to rich dynamics, including pattern formation and even finite-time blow-up. 
We refer the reader to \cite{MR2103197} for the supercritical/subcritical dichotomy in two dimensions. 
It is also well known that, on a bounded domain with homogeneous Neumann boundary conditions, solutions converge to a constant equilibrium whenever the total initial mass (or the chemotactic coupling parameter) is sufficiently small (see, e.g., \cite{MR3294230}).
Very recently, \cite{MR4983182} investigates an attractive logarithmic gas on the torus (including the Keller-Segel system in two dimensions). The authors identify a critical inverse temperature $\beta_s$ at which the spatially uniform state $\mu_{\mathrm{unif}}$ changes from being a locally asymptotically stable equilibrium ($\beta < \beta_s$) to being nonlinearly unstable ($\beta > \beta_s$). Combining this with the classical threshold $\beta_c = 2d$ for boundedness of the free energy, they show that non-uniform stationary states exist whenever $\beta > \min\{\beta_s,\beta_c\}$.
The work \cite{MR2322024} is also very closely related to our result, as it first employed a Lojasiewicz-Simon inequality to establish convergence to equilibrium for the two-dimensional Keller-Segel model.
More precisely, they proved that, for the (parabolic-parabolic) Keller-Segel system with homogeneous boundary conditions, any global and strictly positive solution must converge to a stationary solution, provided that blow-up does not occur. 
Our result shows that, in any spatial dimension, the Keller-Segel system admits convergence to equilibrium whenever the solution remains globally bounded on the torus, see Example \ref{example_KS}.

\smallskip
\noindent\emph{Interacting particle systems, phase transitions, and propagation of chaos.}
The asymptotic behavior has also been extensively studied at the level of the interacting particle (SDE) system. 
More precisely, the associated $N$-particle Langevin dynamics is given by
\begin{equation}\begin{aligned}\label{N_SDE}
&dX_t = -\nabla H_N(X_t)\,dt + \sqrt{2\beta^{-1}} dB_t,
\quad \text{with}\\&
H_N(X^1,\dots,X^N)
= \sum_{i=1}^N V(X^i)
+ \frac{1}{2N}\sum_{i\neq j} W(X^i - X^j),\end{aligned}
\end{equation}
and it was shown in \cite{MR3098681, MR3174216, MR3027901}
that the random variable $X_t$ whose law is $\rho(t)$ converges weakly to a stationary
distribution, provided that the interaction potential $W$ is a convex polynomial (see also \cite{MR4312362} for its variants).
Phase transitions in McKean-Vlasov dynamics are reflected in the existence of multiple stationary solutions. 
The interacting diffusion system \eqref{N_SDE}, written in gradient form with respect to the $N$-particle Hamiltonian $H_N$, admits a unique invariant Gibbs measure $M_N \propto e^{-\beta H_N}$ under quite general assumptions, whereas this uniqueness may fail in the mean-field limit as $N \to \infty$. 
The associated McKean-Vlasov dynamics then typically exhibits multiple critical points, corresponding to several distinct equilibrium states. 
This is naturally interpreted as an illustration of a phase transition at the level of the infinite-particle system.
In \cite{MR4062483}, the authors present a detailed analysis of phase transitions for McKean-Vlasov equations on the torus without any confining potential.

For the \emph{propagation of chaos}, that is, the rigorous derivation of continuum equations such as \eqref{MV} from the $N$-particle dynamics \eqref{N_SDE} via a mean-field limit, we refer the reader to \cite{MR4373166, Sznitman_1991,doi:10.1073/pnas.56.6.1907,ChaintronDiez2022, 10.1215/00127094-2020-0019} and the references therein for more details. In recent years, the Wasserstein gradient-flow and McKean-Vlasov structures have also become central to analyzing the training dynamics of overparametrized neural networks. In the mean-field limit, the parameter evolution under stochastic gradient descent is described by the mean-field Langevin dynamics, which formally constitutes a gradient flow of a free energy functional on the Wasserstein space. This perspective has been instrumental in establishing global convergence guarantees despite the non-convex nature of the loss landscape; see, for instance, \cite{MMN18,CB18,chizatmean,nitanda2022convex,hu2021mean,kim2024transformers}. 

\medskip

\noindent \textbf{Remark on the existence of solutions.} Under our standing assumptions on $V$ and $W$, \textbf{(\ref{assump_A1})}--\textbf{(\ref{assump_A3})}, the McKean--Vlasov equation \eqref{MV} on the torus $\mathbb{T}^n$, $n\ge1$, admits a local-in-time solution for initial data in $L^p$ with $p >\max(\frac{n}{2},1)$ (assuming $W$ has a Newtonian singularity). Although developing a complete well-posedness theory is not our main focus, we briefly sketch the standard existence argument.

We employ the Duhamel principle to rewrite \eqref{MV} in the mild form:
\[
  \rho(t)
    = e^{t\Delta}\rho(0)
    + \int_0^t e^{(t-s)\Delta} \nabla \cdot \bigl(\rho(s)\,[\nabla W * \rho(s) + \nabla V]\bigr)\,ds
    =: \mathcal{T}(\rho)(t).
\]
For initial data $\rho(0) \in L^p(\mathbb{T}^n)$ with $p > \max(\frac{n}{2},1)$, we define the solution space
\[
  X_T := C([0,T], L^p(\mathbb{T}^n)).
\]
Using the $L^r \to L^p$ gradient smoothing estimates of the heat semigroup (i.e., $\|\nabla e^{t\Delta}f\|_p \le C t^{-\frac{1}{2}-\frac{n}{2}(\frac{1}{r}-\frac{1}{p})}\|f\|_r$) combined with the Hardy--Littlewood--Sobolev inequality for the interaction term, one can verify that the time integral converges provided $\frac{n}{2p} < 1$. Consequently, there exist $T>0$ and $R>0$ such that the map $\mathcal{T}$ is a contraction on the ball $\{\rho \in X_T : \|\rho\|_{X_T} \le R\}$. The Banach fixed point theorem then yields a unique local-in-time mild solution in $X_T$.

Furthermore, standard bootstrapping arguments (noting that mild solutions satisfy the equation in the distributional sense, allowing the application of smoothing estimate Lemma~\ref{lemma-regularityrho}) ensure that this mild solution becomes smooth for $t>0$. Once a local-in-time classical solution is obtained, it can be extended up to a maximal existence time $T_{\max}$. In particular, regularity estimates in Lemma~\ref{lemma-regularityrho} imply that if $T_{\max} < \infty$, then
\[
  \limsup_{t \uparrow T_{\max}} \|\rho(t)\|_{L^p} = \infty \quad \text{for all } p >\max(\tfrac{n}{2},1).
\]

For a detailed well-posedness theory for aggregation--diffusion equations under comparable assumptions, see \cite{MR2510126, Bedrossian_2011}. 
In fact, for $n\ge 2$, the Lebesgue space $L^{n/2}$ is scaling-critical for the equation. If $\rho(0)\in L^{n/2}$, one may still obtain local well-posedness under an additional smallness assumption on the critical norm; see, for instance, \cite[Section 3]{MR4983182} for such a result for the logarithmically interacting gas.

\bigskip 
Finally, we also expect that our main idea may be applicable beyond the McKean-Vlasov equation. In particular, it appears reasonable to anticipate its relevance to other equations that can be formulated as Wasserstein gradient flows, such as the thin-film equation \cite{MR2581977}, the quantum drift-diffusion equation \cite{MR2533926}, and the Cahn-Hilliard equation \cite{MR2921215}. We also mention the recent work \cite{isobe2023convergence}, which establishes a {\L}ojasiewicz-Simon inequality for a Wasserstein-type gradient flow via a Hilbert-space lifting, providing a notable alternative 
approach in a setting different from the case considered here. We leave a detailed investigation of these directions for future work. 

\medskip 

In the next subsections, we state our main results precisely and outline the strategy of the proofs by giving main ideas.

\subsection{Main theorems}

Throughout this paper, we assume $n\ge1$ and that smooth confinement potential $V:\mathbb{R}^n/\mathbb{Z}^n \to \mathbb{R}$ and interaction potential $W:(\mathbb{R}^n/\mathbb{Z}^n)-\mathbb{Z}^n  \to \mathbb{R}$ satisfy the following conditions:

\begin{enumerate}[label=\textbf{(A\arabic*)}, ref=A\arabic*]
\item\label{assump_A1}
$V$ is analytic function. i.e., $\Vert \nabla^k V\Vert_{L^\infty}\le C_V^{k+1}k!$  $\forall k\in\mathbb{N}_0$.
\item\label{assump_A2}
 $W$ is centrally symmetric. i.e., $W(z)=W(-z)$.  
\item\label{assump_A3}
There exists $C_W<\infty$ such that 
\begin{equation*}
|\nabla^k  W(z)| \leq \frac{C_W^{k+1} k!}{|z|^{n-2+k}}, \quad \text{for\,} z\in  [-\tfrac12,\tfrac12]^n \text{ for all }k=1,2,\ldots .  %
\end{equation*}	
\end{enumerate}

The assumption~\textbf{(\ref{assump_A3})} states that $W$ has a quantitative analyticity with a (at most) Coulomb-type singularity.  Throughout the paper, we interpret $V$ (resp. $W$) as a $\mathbb{Z}^n$-periodic function on $\mathbb{R}^n$ (resp. $\mathbb{R}^n \setminus \mathbb{Z}^n$). We remark that an assumption similar to \textbf{(\ref{assump_A3})} appeared in the analysis of the nonlocal Cahn-Hilliard equation by Akagi \cite[(H2), p. 2676]{MR3926130}. See Remark~\ref{remark-analyticity}, for the relevance of the analyticity assumption in our approach and related previous results.

Our first main theorem establishes \L ojasiewicz's gradient inequality for energy functionals on the Wasserstein space, which is a Wasserstein analogue of \L ojasiewicz inequality \eqref{167}.
\begin{theorem}[Wasserstein-\L ojasiewicz inequality] \label{thm-main-LS} Let $\mathcal{F}$ satisfy Assumptions \textup{\textbf{(\ref{assump_A1})}--\textbf{(\ref{assump_A3})}}. For $p>n$ and a given stationary 
probability density $\rho_*$ of \eqref{MV} (in the classical sense), there exist $\theta\in [1/2,1)$, $c>0$ and $\sigma>0$ such that if $\Vert \rho -\rho_*\Vert_{W^{1,p}}\le \sigma  $, then 
\be \label{eq-LS}  \Vert \nabla (\log \rho   +  V +  W *\rho)\Vert_{L^2_\rho}  \ge c\, |\mathcal{F}(\rho)-\mathcal{F}(\rho_*)|^\theta .\ee 
 
\end{theorem}

\begin{remark} \label{remark-smoothing}
Under Assumptions \textup{\textbf{(\ref{assump_A1})}--\textbf{(\ref{assump_A3})}}, standard regularity estimates ensure that both stationary solutions and time-dependent solutions (in the sense of distribution) become smooth and strictly positive for $t>0$ provided that they are uniformly bounded in $L^p$ for $p>\max(\frac n 2 ,1)$; see Lemmata~\ref{lemma-regularityrho}, \ref{lemma-lowerbd}, and Remark~\ref{remark-regularityrho}. Therefore, except for those Lemmata, we will treat solution $\rho(x,t)$ as smooth positive classical solution in the remaining of this paper.     
\end{remark}
 Based on the crucial inequality in Theorem~\ref{thm-main-LS}, we characterize the long-time behavior of non-blowing-up solutions as follows. 

\begin{theorem}[global convergence] \label{thm-mainconvergence}
Assume \textup{\textbf{(\ref{assump_A1})}--\textbf{(\ref{assump_A3})}} hold. 
Suppose that $\rho(t)$ is a distributional solution to the McKean--Vlasov equation~\eqref{MV} that does not blow up, in the sense that
\[
\sup_{t\ge 0} \|\rho(t)\|_{L^p}<\infty
\quad \text{for some } p>\max(\tfrac{n}{2},1).
\]
Then the solution converges, as $t\to\infty$, to a stationary solution $\rho_*\in C^\infty(\mathbb{T}^n)$ in the $C^\infty$ topology.
\end{theorem}

We emphasize that stationary solutions to McKean-Valsov \eqref{MV} are in general not unique \cite{Monmarche:2025aa,MR4062483,MR4983182}. 
The theorem above shows that, under our assumptions, bounded trajectory of the flow converges to a single stationary solution uniquely selected by the initial data. We also note that the no-blow-up assumption is essential: there exist choices of $V$ and $W$ satisfying \textup{\textbf{(\ref{assump_A1})}--\textbf{(\ref{assump_A3})}} for which solutions can blow up in finite time; e.g., elliptic-parabolic Keller-Segel equation in Example \ref{example_KS}. In contrast, as demonstrated in Theorem \ref{thm-noblowup}, such blow-ups are ruled out if the interaction potential W is sufficiently regular.

Note also that $C^\infty$ convergence on a compact domain $\Tn$ implies the convergence in Wasserstein distance $d_2$. Moreover, we obtain a rate of convergence. 
\begin{corollary} [rate of convergence c.f. Theorem 3.27 \cite{hauer2019kurdyka}] \label{corollary-rate}Suppose $\rho(t)$ converges to $\rho_*$ as $t\to \infty$ and $\rho_*$ satisfies the \L ojasiewicz inequality \eqref{eq-LS} for $c>0$ and $\theta\in[1/2,1)$. Then there exists $C'<\infty$ such that 
\begin{equation}d_2(\rho(t),\rho_*)\le  C'\begin{aligned}\begin{cases} e^{-\frac {c^2}{2} t} &\text{ if } \theta=\frac12 \\ t^{-\frac{1-\theta}{2\theta-1} } &\text{ if } \theta \in (\frac12,1).

\end{cases} 
\end{aligned} \label{rate}
\end{equation}

\end{corollary}

Finally, in the remark below we compare the Wasserstein-\L ojasiewicz inequality with other convexity conditions and functional inequalities. Here we formally explain why \eqref{eq-LS}, which is weaker than $\lambda$-displacement convexity and log-Sobolev inequalities, is the type of inequality one may expect in nonconvex (or weakly convex) settings, where a critical point of $\mathcal{F}$ may have a degenerate Hessian. 

\begin{remark}[Displacement convexity/LSI vs.\ Wasserstein-Łojasiewicz]\label{remark-LSLS}
The inequality \eqref{eq-LS} becomes stronger as $\theta$ decreases: if it holds with some $\theta$, then it also holds with any $\theta'\in(\theta,1)$ after possibly shrinking $\sigma$. The optimal value $\theta=\tfrac12$ yields exponential convergence of the gradient flow; see Corollary~\ref{corollary-rate}. 

A (nonlinear) log-Sobolev inequality (see, e.g., \cite{Monmarche:2025aa}) for the free energy $\mathcal{F}$ typically reads
\[
I(\rho):=\int \big|\nabla(\log\rho+V+W*\rho)\big|^2\,d\rho \;\ge\; 2\lambda\,\big(\mathcal{F}(\rho)-\mathcal{F}(\rho_*)\big),
\]
which is the optimal case $\theta=\tfrac12$ in \eqref{eq-LS}, since $I(\rho)$ is precisely the squared $L^2_\rho$–norm of the Wasserstein gradient of~$\mathcal{F}$.

Because it is so strong, a log-Sobolev inequality cannot hold in general: for instance, for equations on Euclidean spaces one easily finds functions $f$ for which a critical point $x_0$ does not satisfy a \L ojasiewicz inequality with $\theta=1/2$. Indeed, any gradient flow that converges but not at an exponential rate provides such an example at its limit point. Moreover, such a critical point $x_0$ must necessarily have a degenerate Hessian, in the sense that $\ker \nabla^2 f(x_0)\neq \{0\}$. If $f\in C^2$ and $\nabla^2 f(x_0)$ is nondegenerate, then an elementary argument based on Taylor's theorem yields \eqref{167} with $\theta=1/2$ on a neighborhood of $x_0$. No analyticity assumption is needed in this case.

If $\mathcal{F}$ is $\lambda$-displacement convex for $\lambda>0$, then solutions converge to a unique stationary solution at an exponential rate. In view of Otto's calculus, $\lambda$-displacement convexity formally implies $\mathrm{Hess\,} \mathcal{F}(\rho_*)\ge \lambda I>0$. In this case it is possible to derive \eqref{eq-LS} with $\theta=1/2$ without assuming analyticity of $V$ and $W$. See Corollary \ref{cor-LSnondegenerate}.  In this paper, we computed the Hessian by
\[
\mathrm{Hess\,}\mathcal{F}(\rho_*)[\xi_1,\xi_2]= \langle L\xi_1, \xi_2 \rangle_{L^2_{\rho_*}},
\]
where $L$ is given in \eqref{eq-L} (c.f. \cite[Sec.~3.1]{MR2053570} for corresponding formal computation). From the expression for $L$, one observes that if $V$ and $W$ are (uniformly) convex, then $\langle L \xi, \xi\rangle _{L^2_{\rho_*}}$ is strictly positive for all nonzero regular tangent vector $\xi$, showing that $L$ has only a trivial kernel. i.e., $\rho_*$ is non-degenerate. 

In view of the discussion above, in general nonconvex regime one cannot expect such a log-Sobolev inequality (or local $\lambda$-displacement convexity) to hold at every stationary point; the algebraic or sub-exponential decay observed in that literature is consistent with $\theta>\tfrac12$. See, for example, \cite{MR2357669, MR3035983, Monmarche:2025aa}. In contrast, the Wasserstein-\L ojasiewicz inequality \eqref{eq-LS} with a general exponent $\theta\in[1/2,1)$ precisely captures the weaker coercivity one may still recover near degenerate stationary points.
\end{remark}

\subsection{Main idea}

A very rough outline of the proof of Theorem \ref{thm-main-LS} is as follows. We pull back the functional $\mathcal{F}$ and its gradient to the tangent space of $\mathcal{P}^{ac}_2$ at a fixed base point, which locally models $\mathcal{P}^{ac}_2$ near that point. On this (linear) tangent space, we then aim to establish an (approximate) \L ojasiewicz inequality for the pull-back functional.

\begin{definition}[tangent space and exponential map]
For $\mu_*=\rho_*\,dx$ in $\mathcal{P}^{ac}_2(\mathbb{T}^n)$, we denote by $\mathcal{H}_{\mu_*}$ the Hilbert space obtained by taking the $L^2_{\mu_*}$-completion of the space of gradient vector fields:
\[
\mathcal{H}_{\mu_*}
:= \overline{\{\nabla \Psi : \Psi \in C^\infty(\mathbb{T}^n,\mathbb{R})\}}^{L^2_{\mu_*}(\mathbb{T}^n)}.
\]
For each $X\in \mathcal{H}_{\mu_*}$, we define the map $\exp X:\mathbb{T}^n \to \mathbb{T}^n$ by
\[
\exp X(p) := \exp_p X_p,\qquad p\in \mathbb{T}^n.
\]
We also define $\exp_{\mu_*}: \mathcal{H}_{\mu_*} \to \mathcal{P}^{ac}_2(\mathbb{T}^n)$ by
\[
\exp_{\mu_*}X := (\exp X)_\# \mu_*.
\]
\end{definition}

According to the formalism suggested in \cite{MR1842429}, the space $\mathcal{H}_{\mu_*}$ can be viewed as the tangent space of $\mathcal{P}^{ac}_2$ at $\mu_*=\rho_*\,dx$, that is, $\mathcal{H}_{\mu_*}\cong T_{\mu_*}\mathcal{P}^{ac}_2$. From this point of view, the map
\[
X \longmapsto (\exp X)_\#\mu_* =: \exp_{\mu_*}X
\]
plays the role of the exponential map from $T_{\mu_*}\mathcal{P}^{ac}_2$ to $\mathcal{P}^{ac}_2$. Our approach to Theorem \ref{thm-main-LS} (\L ojasiewicz inequality) starts from representing $\mathcal{F}$ in normal coordinates given by the exponential map $\exp_{\mu_*}$ at the base measure $\mu_*$.

\begin{definition}[pull-back energy]
The pull-back energy $G:\mathcal{H}_{\mu_*}\to \mathbb{R}$ associated with $\mathcal{F}$ at $\mu_*=\rho_*\,dx$ is defined by
\[
G(X) := \mathcal{F}(\exp_{\mu_*}X).
\]
\end{definition}

Unlike $\mathcal{F}$, which is defined on the (formal) Riemannian manifold $\mathcal{P}^{ac}_2$, the functional $G$ has the advantage of being a well-defined functional on the linear space $\mathcal{H}_{\mu_*}$. In the celebrated work of L.~Simon \cite{S0}, \L ojasiewicz's inequality in the finite-dimensional setting was generalized to a class of analytic functionals defined on Banach spaces; see also \cite{MR4129355,MR4199212,MR1609269,MR1986700,MR1800136,MR4129083,MR2226672}. Formally, $(\mathcal{P}^{ac}_2,g)$ is isometric to $(\mathcal{H}_{\mu_*},\exp_{\mu_*}^* g)$. Here $g$ is the (formal) Riemannian metric on $\mathcal{P}^{ac}_2$. This creates a challenge: one has to accommodate \L ojasiewicz's theorem in the Riemannian manifold $(\mathcal{H}_{\mu_*},\exp_{\mu_*}^* g)$. Our crucial observation, however, is that it is in fact sufficient to regard $G$ simply as a functional on the Hilbert space $(\mathcal{H}_{\mu_*}, L^2_{\mu_*})$, whose inner product $L^2_{\mu_*}$ is independent of base point.

\bigskip

Note that $L^2_{\mu_*} = (\exp_{\mu_*}^* g)_0$. The motivation behind this observation is the natural expectation that, as the gradient flow approaches $\mu_*$, the geometry of the space is well approximated by its linearization around $\mu_*$. In this regime, the original Wasserstein gradient flow can be viewed as a small perturbation of the gradient flow of the functional $G$ on the Hilbert space $(\mathcal{H}_{\mu_*},L^2_{\mu_*})$. To derive the inequality on this fixed linear space, we adopt Simon's approach, based on a Lyapunov-Schmidt reduction, and obtain the following.

\begin{theorem}[\L ojasiewicz-Simon inequality on a Banach space]\label{thm-LS-G} Assuming \textup{\textbf{(\ref{assump_A1})}--\textbf{(\ref{assump_A3})}}, let $\rho_*$ be positive probability density in $W^{1,p}$ some $p>n$ and $\mu_*= \rho_*dx$. There exist $\theta\in [1/2,1)$, $c>0$ and $\sigma>0$ such that if $X\in \mathcal{H}_{\mu_*}\cap W^{2,p}$ with $\Vert X \Vert_{W^{2,p}} \le \sigma$, then
\begin{equation}\label{eq-GLSintro}
\Vert \nabla_{L^2_{\mu_*}} G(X) \Vert_{L^2_{\mu_*}}
\;\ge\; c\,\big|G(X)-G(0)\big|^\theta.
\end{equation}
\end{theorem}

In this statement, $\nabla_{L^2_{\mu_*}} G(X)$ denotes the gradient of $G$ at $X$ with respect to the $L^2_{\mu_*}$-inner product; see Definition \ref{definition-gradientmap} for details. Throughout the paper we denote this $L^2_{\mu_*}$-gradient by $\mathcal{M}(X)$. For a derivation of \L oajasiewicz-Simon inequality \eqref{eq-GLSintro}, the main technical steps are (1) to show the map $\mathcal{M}(\cdot)$ is analytic in properly chosen domain and codomain, and (2) the linearization of $\mathcal{M}$ at $0$, $d\mathcal{M}(0)$, which corresponds to the symmetric operator associated with the second variation of $G$ at $0$ is Fredholm and hence has a finite dimensional kernel. They are proved in Section \ref{sec-analyticity} and \ref{sec-fredholm}, respectively. In Section \ref{sec-proofmainthm} we prove that Theorem \ref{thm-LS-G} indeed implies Theorem \ref{thm-main-LS}.

\subsection{Examples and application}

\begin{example}[Keller-Segel equation]\label{example_KS}
On $\mathbb{T}^n$, the parabolic-elliptic Keller-Segel chemotaxis equations take the form
\begin{equation} \label{eq-KS} \begin{cases}
 \begin{aligned}
 \partial_t \rho  &= \Delta \rho -  \nabla\cdot ( \chi \rho \nabla c) \\
0&= \Delta c + (\rho- \bar \rho ),   	
 \end{aligned}
 \end{cases} 
  \text{or\, }  \begin{cases}
 \begin{aligned}
 \partial_t \rho  &= \Delta \rho -   \nabla\cdot (\chi \rho \nabla c) \\
0&= \Delta c -\alpha c + \rho, \, \alpha>0.    	
 \end{aligned}
 \end{cases}
\end{equation}
Here the chemotactic sensitivity $\chi$ and the parameter $\alpha>0$ are constants, and
$\bar{\rho}$ denotes the spatial average of $\rho$, normalized to $1$.
Let $G_0$ be the zero-mean Green function of $\Delta$, that is,
$\Delta G_0 = \delta_0 - 1$ in the sense of distributions and
$\int_{\mathbb{T}^n} G_0(z)\,dz = 0$.
Similarly, let $G_{\alpha}$ denote the Green function of $\Delta - \alpha$,
i.e.\ $\Delta G_{\alpha} - \alpha G_{\alpha} = \delta_0$ in the distributional sense.
Then both equations can be written as a McKean-Vlasov equation with interaction potential
$W = \chi G_{\beta}$ for $\beta = 0$ and $\beta = \alpha$, respectively.
In fact, this potential $W$ satisfies assumptions
\textbf{(\ref{assump_A2})}-\textbf{(\ref{assump_A3})} (for details, see Theorem~\ref{cor-KSeq}), and hence the main theorem applies
to these Keller-Segel equations, possibly with additional advection represented by an
analytic potential $V$. We also remark that the theorem applies to models involving multiple
chemoattractants and chemorepellents:
\[
\rho_t = \Delta \rho + \nabla \cdot (\rho \nabla V)
    - \sum_{i=1}^k \nabla \cdot (\chi_i \rho \nabla c_i),
\]
where each chemical concentration \( c_i \), $i=1,\ldots,k$, satisfies either
\[
0 = \Delta c_i + \rho - \bar{\rho},
\quad \text{or} \quad
0 = \Delta c_i - \alpha_i c_i + \rho.
\]
Here \( \chi_i \in \mathbb{R} \) and \( \alpha_i > 0 \).

\end{example}

\begin{example}[radial interaction potential]
Typical interaction potentials arising in applications are radial. 
On $\mathbb{R}^n$ one writes 
\[
W(x-y)=\omega\bigl(|x-y|^2\bigr)
\]
for some $\omega:[0,\infty)\to\mathbb{R}$. 
On the flat torus $\Tn$ equipped with the geodesic distance $d$, this becomes
\[
W(x-y)=\omega\bigl(d(x,y)^2\bigr),
\quad\text{equivalently,}\quad 
W(z)=\omega\bigl(d(z,0)^2\bigr)
\]
Such $W$ is centrally symmetric, so \textup{\textbf{(\ref{assump_A2})}} holds. If, in addition,
\[
\bigl|\omega^{(k)}(r)\bigr|\;\le\; k!\,C_\omega^{\,k+1}\, {r^{\frac{2-n}{2}-k}},
\qquad k=1,2,\ldots,
\]
then \textup{\textbf{(\ref{assump_A3})}} holds in a neighborhood of $0$. 
For instance, the power-law potentials $$W(z)=\sum_{i=1}^k L_i\,d(0,z)^{\gamma_i},$$ for $\gamma_i\ge 2-n$ and $L_i\in\mathbb{R}$, satisfy this local condition.

Globally on $\Tn$, however, \textup{\textbf{(\ref{assump_A3})}} fails in general as $d(\cdot,0)^2$ is not smooth across the cut locus of $0$, which coincides with the boundary of a fundamental domain. Nevertheless, this lack of smoothness is mild, and the proof of the main theorem adapts with minor modifications. See Theorem~\ref{thm-radialpotential} for its details.
\end{example}

\begin{remark}[on the analyticity assumption] \label{remark-analyticity}
The analyticity assumptions in \textbf{(\ref{assump_A1})} and \textbf{(\ref{assump_A3})} may appear restrictive at first glance; however, they constitute a key ingredient in the derivation of the {\L}ojasiewicz inequality used in this paper. In general, a bounded trajectory of a gradient flow may fail to converge to a critical point (stationary solution), even for a smooth energy functional. We refer to \cite{doi:10.1137/040605266} for counterexamples in the Euclidean setting and \cite{10.4310/jdg/1214459844,POLACIK1996472,POLACIK2002586} for PDE gradient flows.

Analogous to these results, it is an interesting open problem to determine whether there exist non-analytic potentials $V$ and $W$ that admit bounded trajectories with no convergence in our setting. In the specific case where only the confinement potential $\mathcal{V}$ is present (i.e., no diffusion $\mathcal{U}$ and interaction $\mathcal{W} \equiv 0$), one can construct such an example using \cite{doi:10.1137/040605266}, since the solution is given by the push-forward of the initial measure under the flow generated by the vector field $-\nabla V$. 
On a compact Riemannian manifold, solutions to the linear Fokker-Planck equation (involving $\mathcal{U}$ and $\mathcal{V}$ but with $\mathcal{W} \equiv 0$) converge exponentially fast to the unique stationary state $\mu \propto e^{-V(x)} \, dx$, without requiring any convexity assumption on $V$. This follows from a Poincar\'e inequality (spectral gap) or log-Sobolev inequality combined with the bounded perturbation principle; see e.g. Section 4.2 or 5.1-5.2 in \cite{MR3155209}.
\end{remark}

\textit{Organization of the paper:}
The paper is organized as follows: In Section 2, we characterize the $\cH$-gradient map of $G$, denoted by $\cM$, and investigate the analyticity of $\cM$ between function spaces of suitable regularity in a neighborhood of $0$. In Section 3, we show that the second-order linearization of $\cM$, denoted by $L$, is a Fredholm operator.
In Section 4, we establish the {\L}ojasiewicz inequality for $G$, based on the results of Section 2 and 3.
In Section 5, we prove our main results, Theorem \ref{thm-main-LS}, Theorem \ref{thm-mainconvergence} and Corollary \ref{corollary-rate}. In Section 6, we address three applications of the main theorems-namely, an improvement of the main result under stronger assumptions, an application to chemotaxis models, and the study of radial interaction.

\section{Analyticity of gradient map} \label{sec-analyticity}
\subsection{Setting}
In the remaining of this paper, unless otherwise mentioned, we assume $\mu_*=\rho_*dx \in \cP^{ac}_{2}$, $\rho_* \in W^{1,p}$ some $p>n$ and $0<C^{-1} \le \rho_*\le C$ on $\Tn $ for some $C<\infty$.
\begin{definition}[$W^{k,p}$-gradient vector fields]
For $\mu_*\in \cP^{ac}_2$ and $p\in(1,\infty)$,  we denote by $$ \mathring{W} ^{k,p}=  \overline{\{\nabla \Psi \,:\, \Psi \in C^\infty (\Tn ,\mathbb{R})\}}^{W_{\mu_*}^{k,p}(\Tn  )},$$ the Banach space of gradient vector fields equipped with $W_{\mu_*}^{k,p}$-norm
\[\Vert h\Vert_{W^{k,p}_{\mu_*}} = \bigg(\sum _{|\alpha| \le k} \int _{\Tn} |\nabla ^\alpha h |^p d\mu_*\bigg)^{1/p}. \]
If $k=0$, we denote $\mathring{L}^p:=\mathring{W}^{0,p}$.
\end{definition} In view of the Poincare inequality, we have that  if $\Phi\in \mathring W^{k,p}$ for some $k\in \mathbb{N}_0$ and $p\in(1,\infty)$, then $\Phi = \nabla \phi$ for some $\phi \in W^{k+1,p}$. Such a $\phi$ is unique up to addition of a constant.

\medskip

In this section, we consider $G:\mathring{W}^{2,p}\to \mathbb{R}$ defined by 
\[ G(\Phi)= \mathcal{F}((\exp \Phi )_\# \mu_*) .\]
Recall $(\exp \Phi) (x) := x+\Phi(x) $, where we view $x\mapsto x+\Phi(x)$ as a map on $\mathbb{R}^n/\mathbb{Z}^n=\Tn$. By viewing $G$ as a map defined on a subspace of the Hilbert space $\mathcal{H}_{\mu_*}$, we may define the gradient map: 

\begin{definition}[gradient map c.f. Definition 2.1.1  \cite{MR2226672}] \label{definition-gradientmap} For $p\in[2,\infty)$, the gradient of $G$ at $\Phi \in \mathring{W}^{2,p}$ is defined as a vector field in $\mathring{L}^p$, namely $\mathcal{M}(\Phi)\in \mathring{L}^p$, which satisfies \[ \langle  \mathcal{M}(\Phi),\Psi \rangle _{\mathcal{H}_{\mu_*}} = dG(\Phi)[\Psi] \quad \Big(= \frac{d}{dt}\bigg\vert_{t=0} G(\Phi+t\Psi) \Big)  ,\]
for all  $\Psi\in  \mathring{W}^{2,p} .$    
\end{definition}
If a gradient exists at $\Phi$, then it is necessarily unique from the definition. The main theorem in this section shows that $\mathcal{M}$ is an analytic map.

\begin{theorem}[analyticity of gradient map]\label{thm-analyticity2} For $p>n$, there is a neighborhood of $0\in \mathring{W}^{2,p}$, namely $U$, such that the gradient map $\mathcal{M}:\mathring{W}^{2,p}\supset
U \to \mathring{L}^p$ is well-defined and it is analytic on $U$.
    
\end{theorem}
We recall the definition \begin{definition}[analytic function]
Let $X$, $Y$ be Banach spaces and let $U\subset X$ be open.
A map $F:U\to Y$ is said to be \emph{analytic at} $x_0\in U$
if there exist bounded symmetric $k$-linear maps $T_k:X^k\to Y$ for $k\in\mathbb{N}_0$
and a radius $r>0$ such that  $
\sum_{k=0}^\infty \|T_k\|\,r^k < \infty$, 
and
\[
F(x)
= \sum_{k=0}^\infty
T_k[\underbrace{x-x_0, \ldots, x-x_0}_{k\ \text{times}}],
\]
for all  $x\in U$ with  $\|x-x_0\|_X<r$. We say that $F$ is analytic on $U$ if it is analytic at every $x_0\in U$.
\end{definition}

\bigskip
Let us write three energies of $G$ separately as 
\[G(\Phi)=\widetilde{\cU}(\Phi) +  \widetilde{\cV}(\Phi)+ \widetilde {\cW}(\Phi) ,\]
where \begin{equation}\label{eq-Genergies}
\begin{aligned}&\widetilde{\cU}(\Phi):= \int \Big( \log (\rho_*(x)) - \log \det (I_n+ \nabla \Phi(x) )\Big)\,   d\rho_*(x), \\
& \widetilde{\cV}( \Phi):=   \int V(x+\Phi(x))\,  d\rho_*(x) ,    \\
&\widetilde{\cW}(\Phi):= \frac{1}{2} \iint  W( x+ \Phi(x)- y- \Phi(y)  )\,  d\rho_*(x) d\rho_*(y).  
\end{aligned}\end{equation}
In the derivation of $\widetilde{\cU}$, we used $(id+\Phi)_\#\rho_* dx = \frac{\rho_*\circ  ( id+\Phi)^{-1}}{\det (I_n+ \nabla \Phi )\circ ( id+\Phi)^{-1}} dx $.
We will prove the analyticity of gradient map of three energies separately. 

\begin{lemma}[Helmholtz decomposition] \label{lemma-Helmholtz}  For $\gamma \in(1,\infty)$, given $L^\gamma _{\mu_*}$-gradient vector field $X \in L^\gamma _{\mu_*}(\Tn, T\Tn)$ can uniquely be decomposed as
\[X= \Phi+X_0, \]
for some  $\Phi\in \mathring{L}^\gamma $ and  $\mu_*$-divergence free vector field $X_0$, i.e., $\mathrm{div}(\rho_* X_0)=0$ in the sense of distribution. Moreover, $X\mapsto  \Phi$ is bounded linear map from $L^\gamma _{\mu_*}$ to $\mathring{L}^\gamma _{\mu_*}$, which is explicitly given as 
\[\Phi= \nabla \, (\Div (\rho_* \nabla))^{-1} \, \Div (\rho_* X).\]
    \begin{proof} In view of $\rho_*\in L^\infty$,  $J:=\Div (\rho_* \nabla ): W^{1,\gamma } \to W^{-1,\gamma } $ and it is a uniformly elliptic operator which has $\ker J = \mathrm{span}\{ 1\}$ by the strong maximum principle. By the fredholm theory,  $\mathrm{Im\,}J= \{ f \in W^{-1,\gamma}\,:\, \langle 1,f \rangle_{(W^{-1,\gamma})^*\times W^{-1,\gamma}} =0\}$ is closed and   $J^{-1}: \mathrm{Im\,}J \to W^{1,\gamma}/\ker J$ is well-defined linear map which is bounded if the later space is equipped with the quotient norm. This shows $X\mapsto \Phi:=\nabla J^{-1} \Div (\rho_* X)$ is a bounded linear map from $L^{\gamma}_{\mu_*}$ to $\mathring{L}^\gamma$.  It is immediate to check $X-\Phi$ is $\mu_*$-divergence free.

    Finally, to check the uniqueness of decomposition, suppose $X=\Phi_1+X_1=\Phi_2+X_2$, where $\Phi_i=\nabla \phi_i$. Taking divergence, $\Div (\rho_*\nabla(\phi_1 -\phi_2)) =0 $ as an element in $W^{-1,\gamma}$. This shows $\phi_1-\phi_2$ is a constant and hence $\Phi_1=\Phi_2$.
    \end{proof}
\end{lemma}

 It is useful to define the projection $X\mapsto \Phi$ appeared in Lemma \ref{lemma-Helmholtz}.
\begin{definition} [projection onto gradient vector fields]For $\gamma\in(1,\infty)$, we define the projection map
\[P_{\mu_*}:L^\gamma_{\mu_*}\to \mathring{L}^\gamma_{\mu_*},\]
by $P_{\mu_*}(X)= \nabla \, (\Div (\rho_* \nabla))^{-1} \, \Div (\rho_* X) $. 
The map is bounded 
\begin{equation}\label{eq-bddproj}\Vert P_{\mu_*} X \Vert _{L^\gamma_{\mu_*}} \le C {\Vert X\Vert _{L^\gamma_{\mu_*}}}
\end{equation}for some $C=C(\mu_*,\gamma)<\infty $.  
	
\end{definition}

First, we prove $G$ is Gateaux differentiable on a neighborhood of $0$. 
\begin{lemma} \label{lemma-1stvariation}For $p>n$, there is $\sigma_0=\sigma_0(n,p, \rho_*)>0$ such that for $\Phi$, $\xi\in \mathring{W}^{2,p}$ with $\Vert \Phi\Vert_{\mathring{W}^{2,p}} \le \sigma_0$, 
\begin{equation}\label{eq-1stvariation}\begin{aligned} d\widetilde{\cU}(\Phi)[\xi]  &= \int   \Div (\rho_* [(I+\nabla \Phi )^{-1}])(x) \cdot \xi(x) \, dx , \\  
 d\widetilde{\cV}(\Phi)[\xi] &= \int \nabla V(x+\Phi(x)) \cdot \xi (x) \,d\rho_*(x) , \\
d\widetilde{\cW}(\Phi)[\xi]&= \iint \nabla W(x+\Phi(x)-y-\Phi(y)) \cdot \xi(x) \, d\rho_*(y)d\rho_*(x).  \end{aligned}\end{equation} 
	Note that $I+\nabla_i \Phi_j(x)=I +\nabla _j \Phi_i(x)  $ for all $i$ and $j$, and therefore there is no ambiguity in the expression involving divergence.    

\begin{remark}
In \cite[Sec.~3 and~4]{MR2053570}, essentially the same computations as above were carried out in the smooth (formal) setting.
\end{remark}

\begin{proof}
Since $\Phi = \nabla \phi$ and $\xi = \na \zeta$ for some $W^{3,p}$ functions $\phi, \zeta$ and  $p>n$, $\phi\in C^{2,\beta}$ some $\beta(p,n)>0$. Therefore, $\nabla \Phi \in C^{0,\beta}$ and $\nabla_i \Phi_j (x) = \nabla _j \Phi_i (x)$ for all $x\in \Tn$ and $i$, $j$. The same holds for $\xi$. Choose $\sigma_0>0$  so that  \begin{equation}\label{eq-smallness}| \nabla _{i}\Phi_j(x) \eta _i\eta _j  | \le \frac 12 |\eta |^2  ,\end{equation} for all $x\in\Tn$ and $\eta \in \mathbb{R}^n$.  	Since $(I+\na \Phi)$ is invertible (pointwise), we obtain, for all $x\in \Tn$,  
$$\frac{d}{dt}\bigg|_{t=0}\log  (\det A_t) = \Tr(A_0^{-1}\dot{A}_0), \quad A_t=I_n + \na \Phi + t \na \xi.
$$ 
Moreover, there exists $t_0>0$ such that if $|t|<t_0$, then
$$
\Big|\frac{ \log \det(I_n + \na \Phi + t \na \xi) - \log \det (I_n + \na \Phi) }{t} \Big|
\leq C \Tr (|\na \xi|),
$$ uniformly on $x\in \Tn$. 
By the dominated convergence theorem,
\begin{align}
&\frac{d}{dt}\Big|_{t=0} \tcU(\na \phi + t \na \zeta)
 =-\int \Tr \big((I_n + \na^2 \vphi)^{-1} \na^2 \z \big) d\rho_*
 \nn \\
&=-\int A_{ij} \z_{ji} d\rho_*=\int  \p_j(A_{ij}\rho_*)\z_i   dx, 
\nn
\end{align}	
 where $A_{ij}:= [(I_n+\na^2 \phi)^{-1}]_{ij}$. This shows the first result on $d \tcU$.


Next, since $V$ is smooth on $\Tn$, the second result on $d\tcV$ follows directly from differentiating \eqref{eq-Genergies}.

Finally, we investigate $d\tcW$. As \eqref{eq-smallness}, there exists $t_1=t_1(\Vert\xi\Vert_{\mathring{W}^{2,p}})>0$ such that if $|t| <t_1$, then $|\nabla (t\xi)| \le 1/4$ on $\Tn$. For $x,y \in \mathbb{R}^n$, using a geodesic between $x$ and $y$ in $\Tn$, we obtain for $|t|< t_1$
\[ |x+\Phi(x)+t\xi(x)-y-
\Phi(y)-t \xi (y)|_T \ge \frac 14 |x-y|_T. \]
Here $|z|_T= \min_{z'
\in \mathbb{Z}^n} |z+z'|$ (and thus $d(x,y)= |x-y|_T$.)  By the assumption 
\textbf{(A3)}, $|\nabla W(z)| \le C|z|_T^{1-n}$, and hence for all $x\neq y$ in $\Tn$
\[ \begin{aligned} &\big|W(x+ \Phi(x) +t \xi(x) -y -\Phi(y)-t \xi(y)) - W(x+ \Phi(x) -y -\Phi(y))\big| \\
&\le  \Big |\int_{0}^t \nabla W(x+\Phi(x) +s \xi(x)- y-\Phi(y)-s \xi(y))ds \Big||\xi(x)-\xi(y)| \\
&\le C|t| |x-y|_T^{1-n}|\xi(x)-\xi(y)|. 
\end{aligned}\]

Note that $\iint |x-y|_T^{1-n} d\rho_*(x)d\rho_*(y)<\infty$.  The dominated convergence theorem implies that $\frac{d}{dt}\Big|_{t=0} \tcW(\Phi + t \xi)$ is equal to 
\begin{align}\label{eq-5589}
\frac{1}{2}\iint  \na W(x+ \Phi(x) - y- \Phi(y)) \cdot (\xi(x) - \xi(y))  d\rho_*(x) d\rho_*(y),
\end{align}	
which yields the last desired result due to \textbf{(A2)}.



\end{proof}
\end{lemma}

\begin{remark} \label{eq-W2grad}
    The previous proof works even for non-gradient vector fields. In particular, when $\Phi\equiv0$, there holds
    \[\begin{aligned} dG(0)[\xi]&= \int \nabla (\log \rho_* + V+W*\rho_*) \cdot \xi \, \rho_*(x) dx \\
    &=\langle \nabla (\log \rho_* + V+W*\rho_*) ,\xi \rangle_{L^2_{\mu_*}} \end{aligned},\]
    for $\mu_*=\rho_* dx$ with positive density $\rho_* \in W^{1,p}$ and vector field $\xi \in W^{2,p}$. This also shows the formal Wasserstein gradient at $\rho_*dx$ is $\nabla (\log \rho_* + V+W*\rho_*)\in T_{\rho_*dx}\mathcal{P}_2^{ac}$ and \eqref{MV} is the gradient flow of $\mathcal{F}$. 
\end{remark}

 \begin{corollary} For $p>n$ there is $\sigma_0>0$ such that the gradient map $\mathcal{M}: B_{\sigma_0}(0) \subset \mathring{W}^{2,p}\to \mathring{L}^p$ is well-defined. More precisely for $\Phi \in B_{\sigma_0}(0)$, 
 \[\mathcal{M}(\Phi)= \mathcal{M}^{\widetilde{\cU}} (\Phi)+\mathcal{M}^{\widetilde{\cV}} (\Phi)+\mathcal{M}^{\widetilde{\cW }} (\Phi),\] 
\begin{equation} \label{eq-mathcalMs}\begin{aligned} \mathcal{M}^{\widetilde{\cU}} (\Phi)  &= P_{\mu_*} ( \rho_*^{-1} \Div (\rho_* [(I+\nabla \Phi )^{-1}]) ), \\  
\mathcal{M}^{\widetilde{\cV}} (\Phi) &=  P_{\mu_*}( \nabla V(x+\Phi(x)) ) , \\
\mathcal{M}^{\widetilde{\cW }} (\Phi)&= P_{\mu_*} \int \nabla W(x+\Phi(x)-y-\Phi(y))\, d\rho_*(y).  \end{aligned}\end{equation} 
 	
 	\begin{proof} The expression for the gradient is immediate from Lemma \ref{lemma-1stvariation}. It remains to check each term is in $\mathring{L}^p$. In view of \eqref{eq-bddproj}, it suffices to check the function before projection is in $L^p$.  First, $(I+\nabla \Phi)^{-1} \in W^{1,p}$ due to positive definiteness of $I+\nabla \Phi$ and the fact that $W^{1,p}$ is a Banach algebra for $p>n$. This shows $\mathcal{M}^{\widetilde{\cU}} (\Phi)\in \mathring{L}^p$. $\mathcal{M}^{\widetilde{\cV}} (\Phi)\in \mathring{L}^p$ is straightforward because $\nabla V(x+\Phi(x)) \in W^{2,p}$. Finally, \eqref{eq-smallness} implies {$d(x+\Phi(x), y+\Phi(y))\ge \frac12 d(x,y)$}, and therefore 
    \[|\int \nabla W(x+\Phi(x)-y-\Phi(y))d\rho_*(y)| \le C\Vert \rho_*\Vert_{L^\infty} \int_{[-\frac12,\frac12)^n }{|z|^{-n+1}}dz.\]This shows $\int \nabla W(x+\Phi(x)-y-\Phi(y))d\rho_*(y) \in L^\infty$.
 		
 	\end{proof}

 \end{corollary}

\begin{proof}[Proof of Theorem \ref{thm-analyticity2}] It suffices to prove that if $\Vert \Phi\Vert_{\mathring{W}^{2,p}} \le \sigma_0$, then there exists $C=C(n,p,\rho_*,C_V,C_W)<\infty$ (recall $C_V$ and $C_W$ are from \textbf{(\ref{assump_A1})} and \textbf{(\ref{assump_A3})}) such that for each $\Vert \xi_i\Vert _{\mathring{W}^{2,p}}\le1$ and $k\in \mathbb{N}_0$,
\[ \Vert d^k \mathcal{M}(\Phi) [\xi_1,\xi_2,\ldots,\xi_k] \Vert _{\mathring{L}^p} \le C^{k+1} k! .\] We show this property for each of three gradient maps.

\medskip

\noindent \textbf{(i) Analyticity of $\mathcal{M}^{\widetilde{\cV}}$}.

We compute the directional derivative of $\cM^{\tcV}$ at $\Phi$ in the direction of $\xi$. In view of the dominated convergence theorem, 
\begin{equation}\label{eq-dMV}
d \cM^{\tcV}(\Phi)[\xi] = \frac{d}{d\ep}\Big|_{\ep=0} \cM^{\tcV}(\Phi + \ep \xi)
=P_{\mu_*} \na^2 V(id + \Phi)\xi.
\end{equation}
Using \eqref{eq-bddproj} and \textbf{(\ref{assump_A1})}, it follows that 
\begin{equation*}
\Vert d \cM^{\tcV}(\Phi)[\xi] \Vert_{\mathring{L}^p}
\leq \Vert \na^2 V(id + \Phi)\xi  \Vert_{\mathring{L}^p}   \leq C_V^3 2!  \Vert \xi \Vert_{\mathring{L}^p}
\leq  C_V^3 2!.
\end{equation*}
Similarly, we inductively obtain
\begin{align}
&\sup_{\Vert \xi_i \Vert_{\mathring{W}^{2,p}}=1  }\Vert d^k \cM^{\tcV}(\Phi)[\xi_1, ..., \xi_k] \Vert_{\mathring{L}^p} 
\nn \\
&
=  \Big\Vert \frac{\p^k}{\p s_k ... \p s_1}\Big|_{s_1=...=s_k=0} \cM^{\tcV}(\Phi + s_1 \xi_1 + ... + s_k \xi_k) \Big\Vert_{\mathring{L}^p}
\nn \\
&
=\Vert P_{\mu_*} \na^{k+1} V(id + \Phi)\xi_1 ... \xi_k \Vert_{\mathring{L}^p} \leq {C_1^k}C_V^{k+2} (k+1)!.
\nn
\end{align}
Here, $C_1=C_1(n,p,\rho_*)<\infty $ and $\nabla^{k+1}V(id + \Phi)\xi_1 ... \xi_k$ represents
\[x\mapsto d ^k\nabla V(x+\Phi(x)) [\xi_1(x),\ldots, \xi_k(x)]\in \mathbb{R}^n.\] To have the estimate above, we used $W^{1,p}$ is a Banach algebra and thus $\Vert  \Pi_{i=1}^k |\xi_i| \Vert_{L^p} \le\Vert  \Pi_{i=1}^k |\xi_i| \Vert_{W^{1,p} } \le   C_2^{k-1}  \Pi_{i=1}^k \Vert \xi_i \Vert_{W^{1,p}} $. Since $C_1^k C_V^{k+2}(k+1)! \le C^{k+1} k! $ for some $C=C(n,p,\rho_*,C_V)<\infty$, this proves the desired result.

\medskip

\noindent \textbf{(ii) Analyticity of $\mathcal{M}^{\widetilde{\cU}}$}.

 We show that $\cM^{\tcU}$ is analytic on the open ball $B_{\si_0}(0):=\{ \Phi \in \mathring{W}^{2,p} : \Vert \Phi \Vert_{\mathring{W}^{2,p}} < \si_0 \}$, where $\si_0$ will be specified later.

For two square matrices $A,B$, recall that if $A$ is invertible, then $\frac{d}{d\ep}\big|_{\ep=0}(A+\ep B)^{-1}=-A^{-1} B A^{-1}$.With Einstein summation convention, 
\begin{equation} \label{eq-dMU}
d \cM^{\tcU}(\Phi)[ \xi] = 
P_{\mu_*} \Big[  
\rho_*^{-1} \p_j \Big(  
\big(-
(I+\na \Phi)^{-1} \na \xi (I+\na \Phi)^{-1} 
\big)_{ij} \rho_* 
\Big)  
\Big],
\end{equation}
equivalently, by setting $A_{ij}:=[(I+\na \Phi)^{-1}]_{ij}$ and $B_{ij}:=\p_j \xi_i$,
\begin{equation}\label{M_first_derivative}
d \cM^{\tcU}(\Phi)[ \xi] =- 
P_{\mu_*} \Big[  
\rho_*^{-1} \p_j \Big(  
(A_{ik} B_{kp} A_{pj}) \rho_* 
\Big)
\Big].
\end{equation}
By inductively applying similar argument, $\mathcal{M}^{\tcU}$ is differentiable at all orders and 
 \[ d^k \mathcal{M}^{\tcU}(\Phi)[\xi_1,\ldots,\xi_k]=  \sum_{\sigma \in S_k} (-1)^k P_{\mu_*} [\rho_*^{-1} \mathrm{div}(\rho_*   (\Pi _{i=1}^k A B^{\sigma_i })A)],\]
where $B^k_{ij} := \p_{j} (\xi_k)_i $ and  $(\Pi _{i=1}^k A B^{\sigma_i })A= AB^{\sigma_1}AB^{\sigma_2}\cdots AB^{\sigma_k}A$. Since $\rho_*$ is positive function belongs to $W^{1,p}$ and $|S_k|=k!$, for the analyticity, it suffices to check 
\[\Vert ((\Pi _{i=1}^k A B^{\sigma_i })A)_{ij} \Vert _{{W}^{1,p} } \le C^{k+1}, \] for all $1\le i,j\le k$, for some $C<\infty$. In order to see this clearly, let us define 
\[\Vert A\Vert_{W^{1,p},\infty }:= \max_{i,j} \Vert A_{ij}\Vert _{W^{1,p}}.\]
For two matrix valued functions $A,B\in (W^{1,p})^{n^2}$, by considering the matrix multiplication and the fact that $W^{1,p}$ is a Banach algebra, 
\[ \Vert A B \Vert _{W^{1,p},\infty } \le C_{n,p} \Vert A\Vert_{W^{1,p},\infty } \Vert B\Vert_{W^{1,p},\infty }. \] Recall $\Vert B^i \Vert_{W^{1,p},\infty} \le C'  $ as $\Vert \xi \Vert _{\mathring{W}^{2,p}} \le 1$.
Now the result follows since by choosing $\sigma_0$ sufficiently small, we may assume $A=(I+\nabla \Phi)^{-1} $ satisfies $\Vert A \Vert_{W^{1,p},\infty} \le 2$. i.e., 
\[ \Vert AB^{\sigma_1}\cdots AB^{\sigma_k}A\Vert_{W^{1,p},\infty } \le  2^{k+1}C_{n,p}^{2k}C'^k\le (2C_{n,p}^2 C' )^{k+1} . \]

\medskip

\noindent \textbf{(iii) Analyticity of $\mathcal{M}^{\widetilde{\cW}}$}.

As before, let $\si_0 >0$ be small enough so that 
\begin{equation}\label{distance_lower_bound}
d( x+ \Phi (x) , y+ \Phi(y) )
\geq       \frac{1}{2}d(x,y).    
\end{equation}
By taking (Gateaux) derivative of $\mathcal{M}^{\tcW}:\mathring{W}^{2,p}\to \mathring{L}^p$ in \eqref{eq-mathcalMs}, 
\begin{align} \label{eq-dMW}d\cM^{\tcW} (\Phi)[\xi] 
=P_{\mu_*} \int \na^2 W(\cdot + \Phi(\cdot) - y -\Phi(y)) (\xi(\cdot) - \xi(y)) d\rho_*(y).
\end{align}	
The directional derivative above is justified by the same argument as in the proof of Lemma \ref{lemma-1stvariation}:
Note that $|\nabla \xi |\le \sigma $ for some $\tilde \sigma =\tilde \sigma (n,p,\rho_*)$ if $\Vert \xi \Vert_{\mathring{W}^{2,p}}\le 1$. Thus,
$|\nabla W(x+\Phi(x)+t\xi(x)-y-\Phi(y)-t\xi(y))-\nabla W(x+\Phi(x)-y-\Phi(y)) | \le C d(x,y)^{1-n}$, for some $C=C(n,p,\rho_*)<\infty$ for sufficiently small $|t|$. Since $ \int_{\Tn}d(\cdot ,y)^{1-n} d\rho_*(y)<\infty $ is a constant (in particular, bounded), the dominated convergence theorem applies.  Moreover, we estimate the limit using \textbf{(\ref{assump_A3})} and above constants as 
\begin{align}
&  \sup_{\Vert \xi \Vert_{\mathring{W}^{2,p}} =1} 
\Vert  d\cM^{\tcW} (\Phi)[\xi]  \Vert_{\mathring{L}^p}            
\nn  \\
&\leq  C_P
\Big(\int \Big| \int  \frac{ C_W^3 2!}{(2^{-1}d(x,y))^n}\tsi d(x,y)  d\rho_*(y) 
\Big|^p d\rho_*(x)
\Big)^{1/p}
\nn \\
&
=C_{P} C_W^3 2^n \tsi I_{n-1} 2! 
\nn 
\end{align}	
where $C_{P}= \Vert P_{\mu_*}\Vert_{L^p_{\mu_*} \to L^p_{\mu_*} }$ is the operator norm of $P_{\mu_*}$ from \eqref{eq-bddproj} and 
$$
I_\al:=\bigg(  \int \Big( \int \frac{1}{d^{\al}(x,y) } d\rho_*(y) \Big)^p d\rho_*(x)     \bigg)^{1/p} <\infty 
\quad 0<\al<n.
$$

Following a similar argument, higher order Gateaux derivatives are  

\begin{equation} \label{eq-GateauxW}
    \begin{aligned}
        &  d^k\cM^{\tcW} (\Phi)[\xi_1, ..., \xi_k]         = \frac{\p^k}{\p s_k\cdots  \p s_1}\bigg|_{s_i=0} \cM^{\tcW}(\Phi + s_1 \xi_1+\cdots s_k \xi_k  )\\
         &= P_{\mu_*} \int  \na^{k+1}  W ( \cdot + \Phi(\cdot)-y -\Phi(y) )  \prod_{i=1}^k  (\xi_i(\cdot)-\xi_i(y))d\rho_* (y),
    \end{aligned}
\end{equation}
and 
\begin{align}
    &\sup_{\Vert \xi_i \Vert_{\mathring{W}^{2,p}} =1}
\Vert  d^k \cM^{\tcW} (\Phi)[\xi_1, ..., \xi_k]  \Vert_{\mathring{L}^p} 
\nn \\
&\leq 
C_P\Big(\int \Big| \int  \frac{ C_W^{k+2} (k+1)!}{(2^{-1}d(x,y))^{n-1+k}}\tsi^k d^k(x,y)  d\rho_*(y) 
\Big|^p d\rho_*(x)
\Big)^{1/p}
\nn \\  
&
=2^{n-1+k}C_PI_{n-1} C_W^{k+2} \tsi^k  \, (k+1)! \le C^{k+1}k! , 
\nn 
\end{align}
for some $C=C(n,p,\rho_*,C_W)<\infty$.

\medskip

Combining three analyticity estimates, we obtain desired result $$\Vert d^k \mathcal{M}(\Phi) \Vert \le C^{k+1}k! ,$$ for some $C=C(n,p,\rho_*,C_V,C_W)$ if $\Vert \Phi \Vert_{\mathring{W}^{2,p}}< \sigma_0 $. 
\end{proof}

\section{Fredholmness of second variation} \label{sec-fredholm}

The goal of this section is to show that the second variation operator at the origin of $G$ is Fredholm as a map between suitable Banach spaces. More precisely, for positive $\rho_*\in W^{1,p}$ with some $p>n$, this linear map $L:\mathring{W}^{2,p}\to \mathring{L}^p$ is given by 
\begin{equation}\label{characterize_L}
L\xi :=\frac{d}{ds}\Big|_{s=0} \cM (s\xi)=d\mathcal{M}(0)[\xi],
\end{equation}	 and by \eqref{eq-dMV},\eqref{eq-dMU},\eqref{eq-dMW}, it has the following explicit form
\begin{align} \label{eq-L}
&L\xi:= P_{\mu_*} ( -\rho_*^{-1} \Div (\rho_* \na \xi)+ \nabla^2 V\cdot \xi + \mathcal{K}[\xi] ),
\end{align}	where
\begin{equation}\begin{aligned}\label{eq-defK}
 \cK[\xi] &=  \int \na^2 W(\cdot -y)\cdot  (\xi (\cdot)-\xi(y))d\rho_*(y).
\end{aligned}\end{equation}
This form essentially corresponds to the formula formally studied in \cite[Sec. 3.1]{MR2053570}. In particular, the map $L$ is associated with the Hessian of $G$ by 
\[d^2 G(0)[\xi_1,\xi_2] = \langle L\xi_1,\xi_2 \rangle_{\mathcal{H}_{\mu_*}}= \langle \xi_1, L\xi_2 \rangle_{\mathcal{H}_{\mu_*}}.\]

The Fredholmness of $L$ will be crucial in the derivation of \L ojasiewicz-Simon inequality. Regardless of the derivation of $L$, we may view $L$ \eqref{eq-L} as a bounded linear map between  $\mathring{W}^{2,\gamma}$ and $\mathring{L}^\gamma$ for $\gamma\in(1, p]$, (see Lemma \ref{Hessian}). Moreover, we prove the following:

\begin{theorem}[Fredholmness] \label{thm-fredholm} Let $\rho_* $ be a positive probability density lying in $W^{1,p}$ some $p>n$. For $\gamma \in [\frac{p}{p-1},p]$, $L:\mathring{W}^{2,\gamma }\to \mathring{L}^\gamma $ in \eqref{eq-L} is a Fredholm  operator. i.e., $L$ is bounded linear operator which has finite dimensional kernel, closed image and finite dimensional cokernel.
Moreover, 
\begin{enumerate}[label=(\arabic*),leftmargin=2em]
    \item $\ker L \subset \mathring{W}^{2,q}$ for all $q \in[ \frac{p}{p-1},p]$ and thus the kernel does not depend on $\gamma$.
    \item $\mathrm{index\,}L=0$. More precisely, $$\mathrm{Im\,}L= \{f \in \mathring{L}^\gamma : \int f g\, d\rho_*=0 \text{ for all }g\in \ker L\}.$$
\end{enumerate}
	
\end{theorem}
In the remaining of this section, we assume $\rho_* \in W^{1,p}$ for $p>n$. This weak assumption leads some technical difficulties, which one may skip in the first reading; at later Section \ref{sec-proofmainthm}, the result here will be applied for stationary solution $\rho_*$, which is smooth and positive due to elliptic regularity. The proof of Theorem \ref{thm-fredholm} will be given at the end of this section.

\begin{remark} \label{remark-67} For $\Phi$ and $\Psi\in \mathring{W}^{2,\gamma }$ for $\gamma\in (1,\infty)$, we have
    \begin{enumerate}[label=(\arabic*),leftmargin=2em]
 \item      $\int \mathcal{K}[\Phi]\Psi  d\rho_*=\int \mathcal{K}[\Psi ]\Phi  d\rho_* $ and they are equal to
 \[ \frac12 \iint \la \Psi(x)-\Psi(y),  \nabla^2 W(x-y) ( \Phi(x)-\Phi(y)) \ra  \, d\rho_*(x)d\rho_*(y)    .\]
 See Lemma \ref{HLS}.
    \item  If we further assume $\gamma >n$, then
   $$d^2 G(0)[\Phi,\Psi]= \langle L\Phi,\Psi\rangle _{\mathcal{H}_{\mu_*}}=\langle \Phi,L\Psi\rangle _{\mathcal{H}_{\mu_*}}=B[\Phi,\Psi].$$
$B[\Phi,\Psi]$ is a bilinear form associated with $L$ in $\mathcal{H}_{\mu_*}$, which is defined by $
    B [\Phi, \Psi] := \int [ \langle \nabla\Phi,\nabla\Psi \rangle  +  \nabla^2 V[\Phi,\Psi]  + \mathcal{K}[\Phi]\Psi  ] \,d\rho_*$. 
   Here we used $\langle P_{\mu_*} \Phi , \Psi \rangle_{\mathcal{H}_{\mu_*}} = \langle \Phi,\Psi\rangle _{\mathcal{H}_{\mu_*}}$.

    \end{enumerate}
\end{remark}

\begin{lemma} \label{HLS} For $\gamma \in[1,\infty)$, $\mathcal{K}$ in \eqref{eq-defK} is a bounded linear map from $\mathring{W}^{1,\gamma }(\Tn)$ to $L^\gamma(\Tn)$. 

\begin{proof} 
	
Using Assumption~\textbf{(\ref{assump_A2})} and \eqref{eq-defK}, we write
\begin{equation*}\begin{aligned} 
\cK[\Phi](x_1) = & \int \nabla^2 W(x_1 -x_2) (\Phi(x_1)-\Phi(x_2))d \rho_* (x_2).\end{aligned}\end{equation*}
Let $\xi_t$, for $0\leq t \leq 1$, denote the geodesic between $x_1$ and $x_2$ in $\Tn$. In terms of $d(x_1,x_2)$, the distance in $\Tn$, we have 
\begin{equation*}
|\Phi(x_1)-\Phi(x_2)| \leq \int_0^1 |\nabla \Phi(\xi_t)| d(x_1,x_2) dt,    
\end{equation*}
which gives
\begin{align}
&|\cK[\Phi]|
 \leq  \int_0^1 \int  
   \frac{|\na \Phi (\xi_t))|}{d(x_1,x_2)^{(n-1)/\gamma}}  \frac{1}{d(x_1,x_2)^{(n-1)/\gamma'}}
 d\rho_*(x_2) dt  
 \nn \\
& \leq
\int_0^1   \Big
(\int \frac{|\na \Phi(\xi_t)|^\gamma}{d(x_1,x_2)^{n-1}}d\rho_*(x_2)
\Big)^{1/\gamma}
\Big( \int \frac{d\rho_*(x_2)}{d(x_1,x_2)^{n-1}} \Big)^{1/\gamma'}   dt
\end{align} 
By integrating with $\gamma$-th power of both sides and using Minkowski inequality (integral version), we see that
\begin{align}
    &\Big( \int |\cK[\Phi]|^\gamma d\rho_*(x_1) \Big)^{1/\gamma}
    \nn \\
&
\leq
\Big( \int \frac{d\rho_*(x_2)}{d(x_1,x_2)^{n-1}} \Big)^{1/\gamma'}
\int_0^1    \Big( \int \int \frac{|\na \Phi(\xi_t)|^\gamma}{d(x_1,x_2)^{n-1}}d\rho_*(x_2)   d\rho_*(x_1) \Big)^{1/\gamma}
   dt
\end{align} 
Observe that for $t\in[0,1]$, \begin{equation*}\begin{aligned}
 &\int_{\Tn} \int_{\Tn} \frac{|\na \Phi(\xi_t)|^\gamma}{d(x_1,x_2)^{n-1}}d\rho_*(x_2)   d\rho_*(x_1)	\\ &\le \int_{\Tn} \int_{[-\frac12,\frac12)^n} \frac{|\na \Phi(x_1+tz)|^\gamma}{|z|^{n-1}}\rho_*(x_1+tz) dz   d\rho_*(x_1)	\\
 & \le\Vert \nabla \Phi \Vert _{L^\gamma(d\rho_*)}^\gamma \int_{[-\frac12,\frac12)^n} \frac{\Vert \rho_* \Vert_{L^\infty}}{|z|^{n-1}} dz, 
 \end{aligned} 
\end{equation*}
where we used the uniform bound of $\rho_*$ and Fubini's theorem in the last equality. Combining above results, we obtain
\begin{equation*}
\begin{aligned}
&\Big( \int |\cK[\Phi]|^\gamma d\rho_*(x_1) \Big)^{1/\gamma}\le \Vert \rho_*\Vert_{L^\infty}
\Big( \int_{[-\frac12,\frac12)^n} \frac{dz}{|z|^{n-1}} \Big)
\Vert \na  \Phi \Vert_{L^\gamma(d \rho_*)},	
\end{aligned}
\end{equation*}
which proves the desired estimate.
\end{proof}

\end{lemma}

\begin{lemma}\label{Hessian} $L$ in \eqref{eq-L} is bounded linear map from $\mathring{W}^{2,\gamma}$ to $ \mathring{L}^\gamma$ for all $\gamma\in(1,p]$. For $\xi_1\in \mathring{W}^{2,\gamma}$ and $\xi_2\in \mathring{W}^{2,\gamma'}$ with $1/\gamma +1/\gamma'=1$ and $\gamma \in [\frac{p}{p-1},p]$, 
\be \label{eq-657} \int (L\xi_1)\xi_2 \, \rho_* dx = \int \xi_1 (L\xi_2) \, \rho_*dx. \end{equation}
\end{lemma}
\begin{proof}
 Note that, under the assumed condition on $\rho_*$, $\Phi\mapsto \rho_*\nabla \Phi $ is a bounded linear map from $\mathring{W}^{2,\gamma}$ to $W^{1,\gamma}$. The remaining in the proof of the boundedness of $L$ follows directly by Lemma \ref{HLS}  and the boundedness of the projection $P_{\mu_*}$ \eqref{eq-bddproj}. The symmetry in \eqref{eq-657} follows by Remark \ref{remark-67} and an approximation.
\end{proof}	

\begin{lemma}[Calderón-Zygmund]  \label{lemma-CZ} For $p>n$, let $\rho_*$ be positive density in ${W}^{1,p}$. For $\gamma \in [\frac{p}{p-1},p]$, suppose given functions $\Phi$, $h \in \mathring{L}^{\gamma}$ satisfy $L \Phi =h$ when it is tested against smooth gradient vector fields. More precisely, for all $\Psi = \nabla \psi$ with $\psi\in C^{\infty}(\Tn)$, there holds 
\begin{equation}\label{eq-distribution15} \int \Phi (L \Psi) \, \rho_* dx = \int h\Psi \, \rho_* dx  .\end{equation} 
Then $\Phi $ belongs to $\mathring{W}^{2,\gamma}$ and 
\[\Vert \Phi \Vert_{\mathring{W}^{2,\gamma}} \le C(\Vert h\Vert _{\mathring{L}^{\gamma}}+\Vert \Phi \Vert_{\mathring{L}^{\gamma}}), \]
for some $C=C(\gamma,n, p,\rho_*,V,W)<\infty$.    \begin{proof}
Throughout the proof, we will use the following standard Calderon-Zygmund (abbreviated as CZ below) estimate for the Laplacian on $\Tn$: Let $k\in \mathbb{Z}$ and $\gamma \in (1,\infty)$. If a distribution on $\Tn$, namely $\alpha \in \mathcal{D}'(\Tn)$, satisfies 
\[  \alpha  (\Delta f)= \langle f,h\rangle_{(W^{k,\gamma})^*\times W^{k,\gamma}} \quad \text{for all } f\in C^\infty(\Tn)    , \] for some $h\in W^{k,\gamma}$, then $\alpha \in W^{k+2,\gamma}$ and 
\[\Vert \alpha \Vert_{W^{k+2,\gamma}} \le C (\Vert h \Vert_{W^{k,\gamma}} + |\alpha(1)|) ,\] for some $C=C(n,k,\gamma)<\infty $.

\medskip 

Let us first give a proof in the simpler case when $\rho _*\in W^{1,\infty}$ and $\gamma\in(1,\infty)$. In what follows, the constant $C$ depends on $\gamma$, $n$, $\Vert \rho_* \Vert _{W^{1,\infty}}$, $\Vert \rho_*^{-1}\Vert_{L^\infty}$, $V$ and $W$. For $\Psi=\nabla \psi $ with  $\psi \in C^\infty (\Tn)$,  \eqref{eq-distribution15} reads
\begin{equation} \label{eqdist00}\int (-\Delta\Psi -\nabla \log \rho_*\cdot \nabla \Psi )\Phi\rho_*+(\nabla^2V [\Psi,\Phi] + \mathcal{K}[\Psi]\Phi - \Psi h ) \rho_*\, dx  =0 .\end{equation}Here we used $\int \Phi P_{\mu_*} \Phi'  \, d\mu_*  = \int \Phi \Phi'\, d\mu _* $ if $\Phi \in \mathring{L}^\gamma$. We obtain that the following equation holds in the sense of distribution after testing with $\nabla\psi$: for $\Phi=\nabla \phi$, $\phi \in W^{1,\gamma}$, and $h=(h_1,\ldots, h_n)$,
\begin{equation}\label{eq-15660} \Delta ( \nabla_i (\rho_*\nabla _i\phi )) = \nabla_i [ \nabla_j (\nabla _i \rho_*  \nabla_j \phi   )+ \rho_*\nabla^2_{ik} V  \nabla_k \phi    +\rho_* \mathcal{K}[\nabla \phi ]_i -h_i\rho_*] .\end{equation} 
Here $\nabla_i (\rho_* \mathcal{K}[\nabla \phi]_i )=\nabla\cdot (\rho_* \mathcal{K}[\nabla \phi])$ should be understood as 
\[ \langle \psi   , \nabla \cdot (\rho_* \mathcal{K}[\nabla \phi]) \rangle :=  \int  \mathcal{K}[\nabla \psi ] \nabla \phi \, d\rho_*.\]
In view of Lemma \ref{HLS} and $\nabla \phi \in \mathring{L}^\gamma$, $\nabla\cdot (\rho_* \mathcal{K}[\nabla \phi])\in (W^{2,\gamma'})^*=W^{-2,\gamma}.$
The right hand side of \eqref{eq-15660} belongs to an element in $W^{-2,\gamma}$. The CZ estimate yields $\nabla \cdot (\rho_*\nabla \phi ) \in L^\gamma $ and 
\begin{equation}\begin{aligned} \label{eq-1870} \Vert \nabla \cdot (\rho_*\nabla \phi )\Vert_{L^\gamma }&\le C (\Vert \nabla \phi  \Vert _{L^\gamma}+ \Vert h \rho_*\Vert _{W^{-1,\gamma}}) \\ & \le C(\Vert \nabla \phi \Vert _{L^\gamma}+ \Vert h \Vert_{L^\gamma}).\end{aligned}\end{equation} 
Let us denote $\nabla\cdot (\rho_* \nabla \phi)=:f \in L^\gamma $. Testing $\psi/\rho_*$ for $\psi \in C^\infty$,  
\[\int  - \nabla \psi \cdot \nabla \phi + \psi  \rho_*^{-1} \nabla  \rho_* \cdot \nabla \phi - \psi \rho_*^{-1} f   \, dx =0 .\]
This shows that $\phi\in W^{1,\gamma}$ is an (energy) weak solution to \begin{equation} \label{eq-energyweak1}\Delta \phi =- \nabla \phi \cdot \nabla \log \rho_* +f \rho_*^{-1}  .\end{equation} The CZ estimate and \eqref{eq-1870} give that $\phi \in W^{2,\gamma}$ and 
\begin{equation} \label{eq-9850}\begin{aligned}  \Vert \nabla \phi \Vert _{W^{1,\gamma}} & \le C(\Vert \nabla \phi\cdot \nabla \log \rho_* \Vert _{L^\gamma}+ \Vert f \rho_*^{-1}\Vert_{L^\gamma})\\ 
&\le C (\Vert \nabla \phi \Vert _{L^\gamma}+\Vert h \Vert_{L^\gamma }). \end{aligned} \end{equation}

In view of the improved regularity $\nabla \phi \in W^{1,\gamma}$, (instead of \eqref{eq-15660}) we view \eqref{eqdist00} as follows: in distributional sense, $\phi\in W^{2,\gamma}$ satisfies 
\begin{equation}\label{eq560}
   \begin{aligned}
       \Delta(\rho_* \Delta\phi)&= \nabla_i(\nabla_i \rho_* \Delta \phi)-\nabla_i(\nabla_j\rho_* \nabla_{ij}\phi) \\ & \hspace{2cm} +  \nabla_i(\rho_* \nabla^2_{ij}V\nabla _j \phi + \rho_* \mathcal{K}[\nabla \phi]_i - h_i \rho_*).
   \end{aligned} 
\end{equation}
Viewing the right hand side as an element in $W^{-1,\gamma}$, CZ estimate gives
\begin{equation} \label{eq5602}\Vert \rho _*\Delta \phi \Vert _{W^{1,\gamma}} \le C (\Vert \nabla \phi \Vert _{W^{1,\gamma}}+ \Vert h \Vert _{L^\gamma}).\end{equation}
For $\rho_*\Delta \phi =: g \in W^{1,\gamma}$, by the CZ estimate for $\Delta \phi = \rho_*^{-1} g \in W^{1,\gamma}$, \eqref{eq-9850} and \eqref{eq5602}, $\phi \in W^{3,\gamma}$ and we have desired estimate
\begin{equation} \label{eq5603}  \Vert \nabla  \phi \Vert_{W^{2,\gamma}} \le C\Vert g \Vert _{W^{1,\gamma}}  \le C (\Vert \nabla \phi\Vert _{L^{\gamma}}+ \Vert h \Vert _{L^\gamma}).\end{equation}

\medskip

Now suppose $n<p<\infty$ and $\gamma\in [\frac{p}{p-1},p]$. In what follows, the constant $C$ depends on $p$, $\gamma$, $n$, $\Vert \rho_* \Vert _{W^{1,p}}$, $\Vert \rho_*^{-1}\Vert_{L^\infty}$, $V$ and $W$.  In view of the Holder inequality, the right hand side of \eqref{eq-15660} lies in $W^{-2,\gamma_1}$, where $\frac1 {\gamma_1}= \frac{1}{p}+ \frac{1}{\gamma} \in (\frac 1\gamma , 1]$. Suppose $\gamma_1\neq1 $ (i.e., $\gamma>\frac{p}{p-1}$). Following a similar argument as in the previous case, an estimate corresponding to \eqref{eq-9850} gives
\[ \Vert \nabla \phi \Vert_{L^{\gamma_2}}\le C  \Vert \nabla \phi \Vert_{W^{1,\gamma_1}}\le C (\Vert \nabla \phi \Vert _{L^\gamma}+\Vert h\rho_* \Vert_{W^{-1,\gamma_1}}).\]
By Sobolev (or Morrey) embedding, 
\[ \Vert \nabla \phi \Vert _{L^{\gamma_2}}\le C_{\gamma_2,\gamma_1,n}\Vert \nabla \phi \Vert _{W^{1,\gamma_1}} ,\] 
for all $\gamma_2 \in [\gamma_1,\infty)$ which satisfies $\frac 1{\gamma_2} \ge  \frac 1 {\gamma_1} -\frac1 n$. Since $\frac 1 {\gamma_1} -\frac1 n= \frac 1 \gamma + \frac 1 p - \frac 1 n < \frac 1 \gamma $, this in particular implies $\gamma_2$ can be chosen to be larger than $\gamma$ by a definite amount. Now replacing $\gamma$ by $\gamma_2$ and iterating above argument (finitely many times),  we have $\nabla \phi \in W^{1,\gamma}$ and \[\Vert \nabla \phi \Vert _{W^{1,\gamma}} \le C (\Vert \nabla \phi \Vert _{L^\gamma}+ \Vert h\rho_*  \Vert _{W^{-1,\gamma}}).\]
Next, move on to the argument in \eqref{eq560}, we see the right hand side as an element in $W^{-1,\gamma_1}$. Instead of \eqref{eq5602}, we get 
\[ \Vert \rho_* \Delta \phi \Vert _{W^{1,\gamma_1}}\le C(\Vert \nabla \phi \Vert _{W^{1,\gamma}}+ \Vert h \rho_*\Vert _{L^{\gamma_1}}).\] This shows $\Delta \phi \in W^{1,\gamma_1}$ and instead of \eqref{eq5603}, we get 
\[ \Vert \nabla \phi \Vert _{W^{1,\gamma_2}} \le C \Vert\nabla  \phi \Vert_{W^{2,\gamma_1}}\le C (\Vert \nabla \phi \Vert _{L^{\gamma}}+ \Vert h \rho_*\Vert _{L^{\gamma_1}}). \] After iterating this finite times, we obtain desired estimate  
\[ \Vert \nabla \phi \Vert _{W^{2,\gamma}}\le C (\Vert \nabla \phi \Vert _{W^{1,\gamma}}+\Vert h \rho_* \Vert _{L^\gamma})\le C(\Vert \nabla \phi \Vert_{L^\gamma}+\Vert h \Vert_{L^\gamma}).\]
Finally, we address the remaining case $\gamma=\frac{p}{p-1}$ (i.e., $\gamma_1=1$). In this case, the CZ estimate is not applicable. By the duality argument and Sobolev embedding, $L^1$ embeds into $W^{-1,\bar \gamma_1}$ for all $1\le \bar \gamma_1< \frac{n}{n-1}$ (here $\frac{n}{n-1}=\infty$ if $n=1$).  Viewing the right hand side of \eqref{eq-15660}  as an element in $W^{-3,\bar \gamma_1}$, the CZ estimates yields 
\[ \Vert \nabla \phi \Vert _{L^{\bar \gamma_1}} \le C(\Vert \nabla \phi \Vert _{L^\gamma}+\Vert h \rho_* \Vert_{W^{-2,\bar \gamma_1}}).\]
We choose $\bar \gamma_1$ close to $\frac{ n}{n-1} $ so that $\bar \gamma_1 > \gamma = \frac{p}{p-1}$. This reduces the case to the previous case $\gamma \in (\frac{p}{p-1},p]$ and the remaining follows similarly.  \end{proof} 
 \end{lemma}

For equations on bounded domains with suitable boundary conditions or equations on compact manifold with boundary (this is our case), the Freholmness Theorem \ref{thm-fredholm} is a standard result once  Calderon-Zygmund type estimate (e.g., Lemma \ref{lemma-CZ}) is attained and $L$ is (formally) self-adjoint. It follows by an argument which uses the compactness and the duality. Here we give its direct proof for readers' convenience. 
	
	\begin{proof} [Proof of Theorem \ref{thm-fredholm}]

	We already observed that $L$ is a bounded linear map from Lemma \ref{Hessian}.     Suppose the dimension of $\ker L$ is infinite. In view of the Riesz lemma, there exists a sequence of kernel elements $\{u_i\}\subset  \mathring{W}^{2,\gamma}$ such that $\Vert u_i \Vert _{\mathring{L}^\gamma}=1$ and $\Vert u_i - u_j \Vert _{\mathring{L}^\gamma} \ge 1/2$ for all $i\neq j$. By Lemma \ref{lemma-CZ} (Calderon-Zygmund), the sequence is uniformly bounded in $\mathring{W}^{2,\gamma}$ and by the Rellich–Kondrachov theorem there is a subsequence which converges in $\mathring{L}^\gamma$. This is a contradiction.

	Next, we prove the closedness of $\mathrm{Im\,}L$ in $\mathring{L}^\gamma$. Let us consider the quotient Banach space $\mathring{W}^{2,\gamma}/\ker L$, whose norm is given by $\Vert f+\ker L\Vert = \mathrm{inf}_{g\in \ker L} \Vert f-g \Vert$. 	Using Lemma \ref{lemma-CZ} and a compactness argument, we can prove that there exists $C'<\infty$ such that 
	\[ \Vert f +\ker L \Vert_{\mathring{W}^{2,\gamma}/\ker L} \le  C' \Vert L f\Vert_{L^\gamma} \] 
    for all $f\in \mathring{W}^{2,\gamma}$. 	Indeed, if this estimate is not true, there exists a sequence $f_i \in \mathring{W}^{2,\gamma}$ such that $\Vert f_i \Vert_{\mathring{W}^{2,\gamma}}=1$, $\mathrm{dist}_{\mathring{W}^{2,\gamma}}(f_i, \ker L)=1$, and $\Vert L f_i\Vert_{\mathring{L}^\gamma}\le 1/i$. By choosing a further subsequence, we may assume that $f_i$ converges to $f_\infty$ weakly in $\mathring{W}^{2,\gamma}$ and strongly in $\mathring{L}^\gamma$. Since $L f_i \to 0$ in $\mathring{L}^\gamma$ and $L$ is (weakly) continuous, we conclude that $L f_\infty=0$.  By Lemma \ref{lemma-CZ},
    \[ 1= \Vert f_i +\ker L\Vert _{\mathring{W}^{2,\gamma}/\ker L}\le \Vert f_i- f_\infty \Vert_{\mathring{W}^{2,\gamma}} \le C ( i^{-1} + \Vert f_i -f_\infty \Vert _{\mathring{L}^\gamma}),\] 
    and this is a contradiction since the right hand side converges to $0$. To finish the closedness of the image of $L$.  Suppose there exists a sequence $Lf_i$ converging to $h\in \mathring{L}^\gamma$. Using the inequality, we conclude $f_i +\ker L$ is Cauchy in the quotient Banach space and thus it converges to some $f_\infty +\ker L$. By continuity of $L$, $L(f_\infty)=h$.  
	
Next, suppose $g\in \ker L$. Then $Lg=0$ in distribution sense and thus $g\in W^{2,q}$ for all $q\in[ \frac{p}{p-1},p]$ by Lemma \ref{lemma-CZ}. This shows (1). 

Finally, it remains to check the assertion on $\textrm{Im\,}L$ in (2) as the rest assertion $\dim \mathrm{coker\,}L=\dim \ker L$ follows directly from there. Suppose $f\in \mathrm{Im\,}L$ with $f=Lh$. For all $g\in \ker L$, $\int (Lh)g\, d\rho_*=\int h(Lg)\, d\rho_*=0$ by \eqref{eq-657} and this shows one direction of inclusion. On the contrary, suppose $\mathrm{Im\,} L$ is a proper closed subspace of the right hand side. By Hahn-Banach theorem, there is $g\in (\mathring{L}^{\gamma})^*=\mathring{L}^{\gamma'}$ such that $\int (Lh)g\, d\rho_*=0$ for all $h \in \mathring{W}^{2,\gamma}$, but $\int f'g\, d\rho_*=1$ for some function $f'$ in the right hand side. However, by Lemma \ref{lemma-CZ}, such a $g$ lies in $\mathring{W}^{2,p}$ and thus $g\in \ker L$. This implies $\int f' g\, d\rho_*=0$, which is a contradiction.


		
	\end{proof}

\section{\L ojasiewicz inequality } \label{sec-LS}
Once the analyticity of Euler-Lagrange operator and the Fredholmness of the second variation operator are obtained, the  {\L}ojasiewicz-Simon inequality for functional defined on Banach space follows by standard arguments in the literature \cite{MR4129355,MR4199212,MR1609269,MR1986700,MR1800136,MR4129083,MR2226672}. Here we present a straightforward derivation following Jendoubi \cite{MR1609269}, which simplifies the original proof of \cite{S0}.

\medskip

As before, let $\rho_*\in W^{1,p}$ be positive probability density for some $p>n$. For each $\gamma\in [\frac{p}{p-1},p]$, consider $L:\mathring{W}^{2,\gamma}\to \mathring{L}^\gamma $. By Lemma \ref{lemma-CZ}, each element in its kernel belongs $\mathring{W}^{2,q}$ for all $q\in[\frac{p}{p-1},p]$. This, in particular, implies $\ker L\subset \mathring{W}^{2,\gamma}$ is the same for all $\gamma \in [\frac{p}{p-1},p].$ Let us consider the orthogonal projection on $\ker L$ in $\mathcal{H}_{\mu_*}=\mathring{L}^2 $. Explicitly, fix a basis $\{\varphi_i\}_{i=1}^m$ of $\ker L$, which is orthonormal in $\mathcal{H}_{\mu_*}$ and define
\begin{equation}\label{eq-def-kerproj}\Pi u := \sum_{i=1}^m \langle u,\varphi_i \rangle_{\mathcal{H}_{\mu_*}}\varphi_i. \end{equation}
Since $\varphi_i$ and $\rho_* $ are regular enough, this $\Pi$ is indeed a bounded projection from $\mathring{W}^{2,\gamma}$ onto $\ker L$ for all $\gamma\in[\frac{p}{p-1},p]$. From now on we fix this orthonormal basis $\{\varphi_i\}_{i=1}^m$ and $\Pi$. In view of the definition of $\Pi$ and Theorem \ref{thm-fredholm} (Fredholmness), we immediately have the following: 
\begin{lemma}  \label{lemma-isomo} For each $\gamma\in [\frac{p}{p-1},p]$, $\mathscr{L}:\mathring{W}^{2,\gamma } \rw \mathring{L}^\gamma $, defined by
\begin{equation*}
\mathscr{L}u = \Pi u + L u,
\end{equation*}	
is a Banach space isomorphism.  
\end{lemma}

Let us define the map $\cN: \mathring{W^{2,p}} \rw \mathring{L^p}$, by 
$$\cN u=\Pi u + \cM u.
$$ 
Then, by Theorem \ref{thm-analyticity2}, $\cN$ is analytic on a neighborhood of $0$.  Note that $D\mathcal{N}(0)=\mathscr{L}$ is an isomorphism by Lemma \ref{lemma-isomo}.
Thanks to the analytic (local) inverse function theorem, there exist a local analytic inverse map $\Xi$ and two neighborhoods of $0$, namely $W_1(0)\subset \mathring{W}^{2,p}$ and $W_2(0)\subset \mathring{L}^p$, such that  
\begin{align}
&\Xi:   W_2(0) \rw W_1(0) ,                \nn  \\
&\cN(\Xi(f))=f \quad \forall f \in W_2(0),                \nn         \\
&\Xi(\cN(u))=u \quad \forall u \in W_1(0), \nn \\
&\Vert \Xi(f)-\Xi(g) \Vert_{\mathring{W}^{2,p}} \leq C_0 \Vert f-g \Vert_{\mathring{L}^p} \quad \forall f,g \in W_2(0). 
\end{align}	

In the construction above, the point we assumed $p>n$ is to make sure $\Xi$ is an analytic function. In what follows, we choose smaller $W'_1(0)\subset W_1(0)\subset \mathring{W}^{2,p}$ and $W'_2(0)\subset W_2(0)\subset \mathring{L}^p$ to have other estimates in $\mathring{W}^{2,2}$ and $\mathring{L}^2$. Roughly, this step is needed as we eventually need to estimate $\Vert \mathcal{M}u\Vert _{\mathring{L}^2} $ from below in the proof of Theorem \ref{thm-LS-G}. 

\begin{lemma} \label{lemma-M-L}For given $\varepsilon>0$, by choosing $W'_1(0)=B_{\sigma_0} (0)\subset \mathring{W}^{2,p}$ small, we have 
\[\Vert (\mathcal{M}u-Lu)-(\mathcal{M}v-Lv)\Vert _{\mathring{L}^2}\le \varepsilon \Vert u -v\Vert_{\mathring{W}^{2,2} } \]

    \begin{proof}
Let us denote $L_{\Phi}:= d\mathcal{M}(\Phi)$. In particular, $L_0=L$. Note that $L_{\Phi}$ can be viewed as a bounded linear map from $\mathring{W}^{2,2}$ to $\mathring{L}^2$ (c.f. Lemma \ref{Hessian}). We claim that by choosing $\sigma_0>0$ small, 
\[  \sup_{\Phi \in W_1(0)}\Vert L_{\Phi}-L\Vert _{\mathscr{L}(\mathring{W}^{2,2},\mathring{L}^2) }\le \varepsilon.\]
If the claim is true, the result directly follows since
\begin{align*} 
&\|(\cM - L)u - (\cM - L)v \| _{\mathring{L}^2}\\&= \bigg \| \int_0^1 (d \mathcal{M}(v+t(u-v))- L) [u-v] dt \bigg \| _{\mathring{L}^2}\leq   \epsilon \|u-v \|_{\mathring{W}^{2,2}} .
\end{align*}

\medskip 

It remains to prove the claim.  By \eqref{eq-dMV},\eqref{eq-dMU},\eqref{eq-dMW}, we have
\begin{align*}
L_{\Phi} [\Psi] = P_{\mu_*}( -\rho_* ^{-1} \Div ((A (\na \Psi )A )\rho_* ) + \na^2 V (x+\Phi(x))  \cdot \Psi + \cK_{\Phi} [\Psi])
\end{align*}
where $A = (I+\nabla \Phi )^{-1}$ and  $\cK_{\Phi} [\Psi ]:=\int (\na^2 W(x-y + \Phi(x)-\Phi(y)) \cdot (\Psi(x)-\Psi(y)) d\rho_* (y)$. Thus, 
\[\| L_{\Phi} \Psi  - L \Psi \|_{\mathring{L}^2 } \leq  \| P_{\mu_*}(\cE ^{\cU} _{\Phi} [\Psi] + \cE ^{\cV} _{\Phi} [\Psi] + \cE ^{\cW} _{\Phi} [\Psi] ) \|_{\mathring{L}^2} ,  \]
where
\begin{align*}
&\cE ^{U} _{\Phi} [\Psi] = \rho_* ^{-1} \Div (((A \nabla \Psi A) - \nabla\Psi )\rho_* ), \\
&\cE ^{V} _{\Phi} [\Psi] = (\na^2 V (id+\Phi )- \na^2 V )\cdot \Psi, \\
&\cE ^{W} _{\Phi} [\Psi] = \cK_{\Phi} [\Psi ] - \cK [\Psi ]= \int \delta _{\Phi}K(x,y) \cdot (\Psi(x)-\Psi(y))d\rho_*(y),
\end{align*}
and
\begin{align*}
\delta _{\Phi}K(x,y) = \na^2 W(x-y + \Phi(x)-\Phi(y)) - \na^2 W(x-y).
\end{align*}

In view of \eqref{eq-bddproj}, for given $\eps>0$, it suffices to find there exists $\sigma_0>0$ such that if $\Vert \Phi \Vert_{\mathring{W}^{2,p}}\le \sigma_0$, then $\Vert  \mathcal{E}^\cdot_{\Phi}(\Psi)\Vert_{L^{2}}\le \eps \Vert \Psi \Vert_{\mathring{W}^{2,2}}$. We show this for each of three $\mathcal{E}^{\cdot}$s.

\medskip

Recall for all $\eps'>0$, $\| A - I \| _{{W}^{1,p}} \le  \eps' $ if $\| \Phi \|_{{W}^{2,p}}\le \delta=\delta(\eps') $.  Also
\begin{align*}
\|A (\na\Psi) A - \na \Psi \|_{{W}^{1,2}} &\leq  C ( \|A(\na\Psi) (A-I) \|_{{W}^{1,2}} + \|(A-I)\na \Psi \|_{{W}^{1,2}}) \\
&\leq C \| A-I \|_{{W}^{1,p}} \| \na \Psi \|_{{W}^{1,2}} ( \|A \|_{{W}^{1,p}}+1).
\end{align*}
Here we used $\Vert fg \Vert_{W^{1,2}} \le C \Vert f \Vert _{W^{1,p}}\Vert g \Vert _{W^{1,2}}$ for $p>n$, which can immediately be shown using Morrey and Sobolev inequalities. This proves that for all $\eps$, there exists $\sigma_0 >0$ such that if $\Vert \Phi \Vert_{\mathring{W}^{2,p}}\le \sigma_0 $, then $\Vert \cE_{\Phi}^{U} [\Psi]\Vert_{{L}^2}\le \eps  $ for all $\Vert \Psi \Vert_{\mathring{W}^{2,2}} \le 1$. 

\medskip

Next, in view of the mean value theorem, a bound on $|\nabla^3 V|$ from \textbf{(\ref{assump_A1})} and Sobolev embedding, 
\[ \|\na ^2 V (x+\Phi ) - \na^2 V(x) \|_{L^\infty } \le C \Vert \Phi \Vert_{\mathring{W}^{2,p}}.\]
This shows the desired smallness of $\Vert  \cE_{\Phi}^{V}[\Psi]\Vert _{{L}^2}$ for all $\Vert \Psi\Vert_{\mathring{W}^{2,2}}\le 1$ if $\Vert \Phi \Vert_{\mathring{W}^{2,p}}$ is sufficiently small.

\medskip 

Finally, to estimate $\mathcal{E}^{W}$, it suffices to have a pointwise estimate on $\delta _{\Phi}K$, namely $|\delta _{\Phi}K(x,y)|\le C \Vert \Phi\Vert _{\mathring{W}^{2,p}} d(x,y)^{-n}$ for $\Vert \Phi\Vert_{\mathring{W}^{2,p}}\le \sigma_0$. Note that $\Psi \mapsto \int d(x,y)^{-n} (\Psi(x)-\Psi(y)) d\rho_*(y) $ is a bounded linear map from $\mathring{W}^{1,2}$ to $L^2$: see the proof of Lemma \ref{HLS}. Thus, 
\[\Vert  \cE_{\Phi}^{W}[\Psi] \Vert 
_{L^2} \le C \Vert \Phi \Vert_{\mathring{W}^{2,p}}\Vert \Psi \Vert_{\mathring{W}^{1,2}}.\] For the pointwise estimate on $\delta _{\Phi}K$, assume $\sigma_0 $ is small so that $|\nabla \Phi|\le 1/2$. This guarantees that a geodesic $\{\gamma(t)\}_{t\in[0,1]}$ between $x-y$ and $x-y+\Phi(x)-\Phi(y)$ in $\Tn$ satisfies $ d(0,\gamma(t))\ge \frac12 d(x,y)  $, for all $t\in[0,1]$. Moreover, observe
\begin{align} \label{eq-KM3}
\delta _{\Phi}K(x,y) = \int_0^1 \na^3 W( \gamma(t) ) \cdot \dot \gamma(t) dt,
\end{align}
and $\int_0^1|\dot\gamma(t)|dt\le |\Phi(x)-\Phi(y)|\le C d(x,y) \Vert \Phi \Vert _{\mathring{W}^{2,p}}$.  By \textbf{(\ref{assump_A3})} and $d(0,\gamma(t))\ge d(x,y)/2$, $|\nabla^3W(\gamma(t))|\le Cd(x,y)^{-1-n}$ and this implies the desired estimate on $\delta _{\Phi}K$. This finishes the proof of the claim and the lemma.
\end{proof}
\end{lemma}

\begin{lemma} By choosing $W'_1(0)=B_{\sigma_0}(0)\subset \mathring{W}^{2,p}$ and $W'_2(0)=B_{\sigma_0'}(0)\subset \mathring{L}^p$ sufficiently small, there hold  
\begin{align}\label{LS_key_2}
\Vert \cM(u)- \cM(v) \Vert_{\mathring{L}^2}& \leq C \Vert u-v \Vert_{\mathring{W}^{2,2}} \quad &&\forall u,v \in W'_1(0),\\ 
\| \Xi (f)- \Xi (g) \|_{\mathring{W}^{2,2}}& \le C  \|f-g \|_{\mathring{L}^2} \quad &&\forall f,g \in W'_2 (0). \label{LS_key_3}
\end{align}
\begin{proof}
Let $u=\Xi(f)$ and $v= \Xi(g)$. i.e.,
\begin{align*}
\cN u =\Pi u + Lu + (\cM - L)u =f, \\
\cN v = \Pi v + Lv + (\cM - L)v = g.
\end{align*}
Note that $(L+\Pi ) (u-v) = (f-g) - ((\cM - L)u-(\cM -L)v)$. Thus by taking $C_1 :=  \|(L+ \Pi )^{-1} \|_{\cL (\mathring{L}^2 , \mathring{W}^{2,2})} <\infty $ (by Lemma \ref{lemma-isomo}), 
\begin{align*}
\|u - v \|_{\mathring{W}^{2,2}} &\le C_1  ( \| f-g \|_{\mathring{L}^{2}} + \| (\cM - L)u - (\cM - L)v) \|_{\mathring{L}^2}). 
\end{align*}
By Lemma \ref{lemma-M-L}, we obtain \eqref{LS_key_3} for $C=2C_1$ for small $\sigma_0'>0$. The inequality \eqref{LS_key_2} is a direct consequence of Lemma \ref{lemma-M-L} and the fact that $L:\mathring{W}^{2,2} \to \mathring{L}^2$ is bounded as observed in Lemma \ref{Hessian}.

\end{proof}
\end{lemma}

\begin{proof} [Proof of Theorem \ref{thm-LS-G}] Recall $\{\varphi_i\}$ are $\mathcal{H}_{\mu_*}$ orthonormal basis of $\ker L$ and we defied the projection $\Pi$ by \eqref{eq-def-kerproj}. Let us define reduced energy $\Gamma$ as the pull-back of $G$ through $\Xi$: i.e., $\Gamma:\mathbb{R}^m \to \mathbb{R} $ with 
$\Gamma(\xi):= G\Big (\Xi \Big (\sum_{i=1}^m \xi^i \varphi_i \Big )\Big ) $. Note 
\begin{equation}\label{eq-nablaGamma}
    \frac{\partial}{\partial{\xi_j}}\Gamma(\xi)= \langle \mathcal{M}(\Xi(\xi^i\varphi_i)),d\Xi (\xi^i \varphi_i)[\varphi_j]\rangle_{H_{\mu_*}} 
\end{equation}
and this, in particular, implies $\Gamma$ is an analytic function on a neighborhood of $0\in \mathbb{R}^m$ as $\mathcal{M}$ and $\Xi$ are analytic near $0$.

Suppose $\Pi u = \sum_i \xi^i\varphi_i $ and $u\in B_{\sigma_0}(0)\subset W'_1(0)$, $\sigma_0\in (0,\sigma)$. Then   
\begin{equation}\label{LS_final_1} \begin{aligned} |\nabla \Gamma(\xi)|&\underset{(\ref{eq-nablaGamma})}{\le}   C \Vert \mathcal{M}(\Xi(\xi^i \varphi_i))\Vert_{\mathring{L}^2}= C\Vert \mathcal{M}(\Xi(\Pi u ))\Vert_{\mathring{L}^2}\\
&\underset{(\ref{LS_key_2})}{\le}  C\| \cM (u) \|_{\mathring{L}^2}+C\| (\Xi (\Pi u ))  - u\|_{\mathring{W}^{2,2}}  \\
&\underset{(\ref{LS_key_3})}{\le}C   \|\cM(u)\|_{\mathring{L}^2}.
\end{aligned}\end{equation}
Here $\sigma_0$ is chosen so that $\Xi(\Pi u) \in W_1'(0)$ and $\Pi u$, $\Pi u +\mathcal{M}u \in W'_2(0)$.

By the classical \L ojasiewicz inequality for analytic function on finite Euclidean space \cite{Loj,lojasiewicz1965ensembles,Loj3}, on a neighborhood of $0\in \mathbb{R}^m$, 
\begin{equation} \label{eq-LSGamma}
      |\Gamma(\xi)-\Gamma(0)|^{\theta} \le  C |\nabla \Gamma(\xi)|.
\end{equation}
By choosing $\sigma_0$ small, we may assume \eqref{eq-LSGamma} holds for $\xi\in \mathbb{R}^m $ with  $\Pi u= \sum \xi^i \varphi_i$. 

Recalling $\Gamma(\xi)=G(\Xi (\Pi u))$, $\Gamma(0)=G(0)$, and the concavity of $x\mapsto x^{\theta}$,  
\begin{equation} \label{eq-166}\begin{aligned}|\Gamma(\xi)-\Gamma(0)|^{\theta}& \ge \big (|G(u)-G(0)|^{\theta} - |G(u)-G(\Xi (\Pi u))|^\theta \big )  \end{aligned}\end{equation}
We claim that 
\begin{equation}\label{eq-Gclaim} |G(u)-G(\Xi(\Pi u))| \le C \Vert \mathcal{M}u\Vert _{\mathring{L}^2}^2 .  \end{equation} In view of \eqref{LS_final_1}--\eqref{eq-Gclaim} with $\theta \ge 1/2$, there exists $C<\infty$ such that  
\[  |G(u)-G(0)|^{\theta} \le C  \Vert \mathcal{M} u\Vert_{\mathring{L}^{2}}  , \]
for all $u \in  B_{\sigma_0}(0) \subset \mathring{W}^{2,p}$, which is desired inequality \eqref{eq-GLSintro}.

It remains to show the claim \eqref{eq-Gclaim}. This follows by combining
\begin{align*}
   G(\Xi (\Pi u))-G(u) 
   &=  \int _0 ^1 \Big < \cM \big(u + t (\Xi (\Pi u) - u)\big ) , \Xi (\Pi u) - u \Big >_{\cH} dt ,
\end{align*}
with two estimates 
\begin{align}
\Vert \cM \big(u + t (\Xi (\Pi u) - u)\big )\Vert_{\mathring{L}^2} \underset{(\ref{LS_key_2})}{\le} \Vert\cM  u \Vert_{\mathring{L}^2} + C \Vert\Xi (\Pi u) - u \Vert_{\mathring{W}^{2,2}} , \nn      
\end{align} and 
\begin{align}
   \Vert\Xi (\Pi u) - u  \Vert_{\mathring{L}^{2}}\le   \Vert\Xi (\Pi u) - u  \Vert_{\mathring{W}^{2,2}} \underset{(\ref{LS_key_3})}{\le} C \Vert \mathcal{M}u \Vert_{\mathring{L^2}}  .\nn 
\end{align}

\end{proof}
We note the analyticity assumption is not needed in the \L ojasiewicz inequality if one a priori knows $\rho_*$ is non-degenerate.  
\begin{corollary} [non-degenerate critical point] \label{cor-LSnondegenerate}Instead of \textup{\textbf{(\ref{assump_A1})}--\textbf{(\ref{assump_A3})}}, suppose $V\in C^3$, $W$ is symmetric $W(z)=W(-z)$ and $|\nabla^k W(z)|\le C|z|^{2-n-k}$ for $z\in [-1/2,1/2]^n$, $k=1$, $2$, and $3$. For a given positive probability density $\rho_*\in W^{1,p}$ some $p>n$, if $\ker L =\{0\}$, then Theorem \ref{thm-LS-G} holds with $\theta=1/2$. 
\begin{proof} In the proof of Theorem \ref{thm-LS-G}, the analyticity \textbf{(\ref{assump_A1})} and \textbf{(\ref{assump_A3})} are used to derive $\Gamma(\xi)$ is analytic and therefore \L ojasiewicz inequaltiy \eqref{eq-LSGamma} holds. If the kernel is trivial \eqref{eq-LSGamma} is not needed as $\Xi:\ker L=\{0\}\to \mathbb{R}$ becomes a trivial map. In the previous proof, \eqref{eq-Gclaim} directly implies the desired result. The $C^3$ regularity assumptions on $V$ and $W$ are required to follow the rest of argument in Section \ref{sec-fredholm} and \ref{sec-LS}.
\end{proof}
\end{corollary}

\section{Proofs of main theorems}  \label{sec-proofmainthm}
\subsection{Proof of Theorem \ref{thm-main-LS}}
We prove \L ojasiewicz inequality for $\cF$.

\begin{lemma}\label{lemma-IFT2} Let $\mu_*=\rho_*dx\in \cP^{ac}(\Tn )$ satisfy 
\begin{enumerate}
    \item $0 < C^{-1} < \rho_*  < C $ for some $C<\infty$, and
    \item $\rho_* \in W^{2,p}$ for some $p>n$.
\end{enumerate} 
Then there exist $\varepsilon>0$ and $C'<\infty$ depending on $n$, $p$, $C$, and $\Vert \rho_*\Vert_{W^{2,p}}$ such that if $\mu_1=\rho_1dx \in \cP^{ac}(\Tn)$ satisfies $\Vert \rho_*-\rho_1\Vert_{W^{1,p}}\le \varepsilon$, then the gradient vector field $\Phi $ induced by the optimal transport map from $\mu_*$ to $\mu_1$ satisfies 
\[
\Vert \Phi \Vert _{\mathring{W}^{2,p}} \le C \Vert \rho_*-\rho_1 \Vert _{W^{1,p}}. 
\]

\begin{proof}

Consider the map $J: \mathring{W}^{2,p} \to W^{1,p}\cap \{\rho\,:\, \int_{\Tn} \rho dx=1\} =:M$ defined by 
\[ J(\Phi)(x)= \rho_*(x+\Phi(x)) \det (I+\nabla \Phi(x)). \]
Note $J(\Phi)\in W^{1,p}$ since $W^{1,p}$ is a Banach algebra for $p>n$. Moreover using $\rho_*\in W^{2,p}$, we also obtain that $J$ is a $C^1$-map.

Next, observe that 
\[dJ(0)[\Phi]= \nabla \cdot ( \rho_*\Phi).\]
We claim that $dJ(0)$ is an isomorphism between $\mathring{W}^{2,p}$ and $T_{\rho_*}M=W^{1,p}\cap \{h \,:\, \int_{\Tn} h dx =0\}$. Recall that $W^{3,p}\cap \{\int_\Tn \phi dx =0\}$ is isomorphic to $\mathring{W}^{2,p}$ by the gradient $\phi \mapsto \nabla \phi $. If $\Phi =\nabla \phi $, $dJ({0})[\nabla \phi] = \nabla \cdot (\rho_* \nabla \phi )$. As an elliptic operator from $W^{3,p}$ to $W^{1,p}$, the operator $L\phi:=\nabla \cdot (\rho_* \nabla \phi )$ has $\ker L=\mathrm{span}\{1\}$, which follows from the strong maximum principle. By the Fredholm theory together with elliptic regularity (Calderon-Zygmund), $\mathrm{Im\,}L= W^{1,p}\cap \{ \int h\cdot 1 \, dx =0\}= T_{\rho_*}M$. This shows $dJ(0):\mathring{W}^{2,p}\to T_{\rho_*}M$ is a bounded bijective linear map (which has bounded inverse by the open mapping theorem.) By the inverse function theorem\footnote{Since $M$ is an affine space $M=1+ \{h\in W^{1,p}\,:\, \int_{\Tn} h dx =0\}$, one does not need to use a version of inverse function theorem whose target is a submanifold of $W^{1,p}$.}, we obtain that $J$ is a $C^1$-diffeomorphism from a neighborhood of $0\in \mathring{W}^{2,p}$ onto a neighborhood of $J(0)\in \rho_* \in M$. Now for $\rho_1\in M$ sufficiently close to $0$, $J^{-1}(\rho_1)=  \tilde \Phi $ satisfies 
\[ \Vert \tilde \Phi -0 \Vert _{\mathring{W}^{2,p}} \le C\Vert \rho_1-\rho_*\Vert _{W^{1,p}} .\]

Finally observe that if $\rho_1$ is close to $\rho_*$ and hence $I+\nabla \tilde \Phi\ge \frac 12 I $, then $x\mapsto x+\tilde \Phi(x) $ is the optimal transport map from $\rho_1 dx$ to $\rho_* dx$. It is direct by calculation that $(id+\tilde \Phi)_\# \rho_1 dx = \rho_* dx $ and the fact that the gradient of a $\frac12 d^2$-convex function which pushes $\mu_1$ to $\mu_*$ is unique, and moreover, it is optimal by \cite{brenier1991polar,mccann2001polar}. The optimal transport from $\rho_*dx $ to $\rho_1 dx $ is $(id +  \tilde \Phi)^{-1}$ and it is $id+\Phi$ for some $\Phi \in W^{2,p}$. One may directly check $\Vert \Phi\Vert_{W^{2,p}} \le C \Vert \tilde \Phi \Vert _{W^{2,p}}$ and this finishes the proof.

\end{proof}\end{lemma}

\begin{proof}[Proof of Theorem \ref{thm-main-LS}] Suppose $\Phi$ is the gradient vector field induced by the optimal transport from $\mu_*=\rho_*dx$ to $\mu=\rho dx$. In particular $(\exp \Phi )_\# \mu_*=\mu $. Observe $|G(\Phi)-G(0)|= |\mathcal{F}(\mu )-\mathcal{F}(\mu_*)|$ and the stationary solution $\rho_*$ is smooth and positive (see Remark \ref{remark-regularityrho} and Lemma \ref{lemma-regularityrho}.)

In view of Lemma~\ref{lemma-IFT2} and Theorem~\ref{thm-LS-G}, it suffices to prove  
\begin{equation}\label{eq-gradientcomparison}
\Vert \mathcal{M}(\Phi)\Vert_{\mathcal{H}_{\mu_* }}
   \le \Vert \nabla (\log \rho +V+W* \rho)\Vert_{\mathcal{H}_{\mu} }.
\end{equation}
On the flat torus \(\Tn = \mathbb{R}^n / \mathbb{Z}^n\), the exponential map is given by  
\[
\exp_{[x]}(\Phi) = [x+\Phi],
\]
where \(x\in \mathbb{R}^n\), \([x]=x+\mathbb{Z}^n\in \mathbb{R}^n/\mathbb{Z}^n\), and we identify \(T_{[x]}\Tn\) with \(\mathbb{R}^n\).  
For simplicity, we will write this below as \(\exp_x(\Phi)=x+\Phi\).

For a given gradient vector field \(\Psi=\nabla \psi\in \mathring{W}^{2,p}\), observe that  
\[G(\Phi + \varepsilon \Psi) 
  = \mathcal{F}\big( \big( (id+\varepsilon \Psi\circ (id+\Phi)^{-1})
     \circ (id+\Phi) \big)_\# \mu_* \big),\]
and also 
\[\big( (id+\varepsilon \Psi\circ (id+\Phi)^{-1})\circ (id+\Phi) \big)_\# \mu_*
   = (id+\varepsilon \Psi\circ (id+\Phi)^{-1})_\# \mu.\]
Differentiating at \(\varepsilon=0\) using Lemma \ref{lemma-1stvariation} and Remark \ref{eq-W2grad},  
\[\langle \mathcal{M}(\Phi), \Psi \rangle_{\mathcal{H}_{\mu_*}} 
  = \langle  \nabla (\log \rho +V+W* \rho), \,proj_{\mathcal{H}_\mu}
     \Psi\circ (id+\Phi)^{-1}  
    \rangle_{\mathcal{H}_\mu}.\]
Since 
\[\big\| proj_{\mathcal{H}_\mu}(\Psi\circ (id+\Phi)^{-1}) \big\|_{\mathcal{H}_\mu} 
  \le \big\| \Psi\circ (id+\Phi)^{-1}\big\|_{L^2_\mu} 
  = \|\Psi\|_{\mathcal{H}_{\mu_*}},\] we approximate $\mathcal{M}(\Phi)$ by $\Psi$ in $\mathring{L}^2$ and this implies \eqref{eq-gradientcomparison}.
\end{proof}

\subsection{Proof of Theorem \ref{thm-mainconvergence}} \label{subsec-convergence}
We need a series of lemmata. In the first two lemmata, we address regularity improvement of solutions to \eqref{MV}.

\begin{lemma}[smoothing estimate] \label{lemma-regularityrho}  For $p>\max(\frac{n}{2},1) $, suppose that  $\rho \in L^{\infty}([0,1],  L^p(\Tn ))$ is a solution to \eqref{MV} (in distribution sense) with Assumptions~\textup{\textbf{(\ref{assump_A1})}--\textbf{(\ref{assump_A3})}}. For all $\delta>0$ and $k\in \mathbb{N}_0$, there exists $C=C( \Vert \rho \Vert_{L^\infty_t L^p_x},\delta, k,p,n,V,W) $ such that
\begin{equation}\label{eq-smoothinglem1} \sup_{t\in [\delta,1] }\Vert \rho(t)\Vert _{W^{k,p}(\Tn)}\le C<\infty . \end{equation}
Consequently, differentiating \eqref{MV} in time $t$, the estimate above yields
\begin{equation}\label{eq-smoothinglem2} \Vert \rho \Vert _{W^{k,p}( \Tn\times [\delta ,1])} \le C <\infty.\end{equation}

\end{lemma}
\begin{proof}
The proof follows by Calderon-Zygmund estimate ($L^p$ theory) and a standard bootstrapping argument. The condition $p>n/2$ is  required to have improved integrability of solution after one iteration (and hence later on). Throughout the proof we use the following version of parabolic Calderon-Zygmund estimate: Let $u(x,t)$ be a (weak) distributional solution to $\partial_t u - \Delta u =f$ on $ [0,1] \times \Tn $. For $1<\ell_1 ,\ell_2 <\infty$ and $k\in \mathbb{Z}$, there holds
\[\begin{aligned} & \Vert \partial_t u\Vert _{L^{\ell_1} ([1/2 ,1],W^{k,\ell_2 })} + \Vert  u \Vert _{L^{\ell_1}([1/2 ,1],W^{k+2,\ell_2 })} \\ & \hspace{3cm}   
\le C_{k,\ell_1,\ell_2} (\Vert f \Vert _{L^{\ell_1}([0 ,1],W^{k,\ell_2 })}+\Vert u \Vert _{L^{\ell_1}([0 ,1],W^{k,\ell_2})}). \end{aligned}  \]

\medskip

It suffices to show one step improvement: for given $k\ge0$, 
\[ \Vert \rho \Vert _{L^\infty ([\delta,1], W^{k+1,p})}\le C( \Vert \rho \Vert _{L^\infty ([0,1], W^{k,p})} ,\delta,k,p,n,V,W).\]
In below, we omit the dependence of constants in $V$ and $W$. Consider the case $n/2<p<n$. Observe the only nonlinear term in \eqref{MV} is 
\[\nabla \cdot ( \rho (\nabla W*\rho )).\] Recalling 
$|\nabla W(z)|\le C/|z|^{n-1}$, if $\rho \in W^{k,p}(\Tn)$, then the Hardy-Littlewood-Sobolev inequality implies 
\begin{equation} \label{eq-r} \Vert \nabla W * \rho\Vert_{W^{k,r}(\Tn)} \le C \Vert \rho \Vert _{W^{k,p}(\Tn)},\end{equation}where $r(p)$ satisfies $\frac 1 r= \frac 1 p - \frac 1 n$. By the Holder inequality, $\rho (\na W*\rho) \in W^{k,q}(\Tn)$, for $\frac {1} q =\frac 1 p+\frac 1r$. Combining above, for $\frac 1 q+ \frac 1 n = \frac 2 p$, 
\[ \Vert \nabla \cdot ( \rho (\nabla W * \rho)) \Vert_{L^ \ell  _tW^{k-1,q}_x} \le C \Vert \rho  \Vert _{L^{2\ell}_t W^{k,p}_x}^2. \]
Now, viewing $f:=\nabla \cdot (\rho (\nabla W * \rho))+\nabla\cdot (\rho\nabla V)$ as given function in $L^\ell _t W^{k-1, q}_x$, the parabolic Calderon-Zygmund estimate yields 
\begin{equation} \label{eq-837}\begin{aligned} 
\Vert \rho \Vert _{L^\ell ([\delta,1], W^{k+1,q})} 
&\le C( \Vert \rho \Vert_{L^{2\ell}([0,1], W^{k,p})}^2+ \Vert \rho \Vert_{L^\ell ([0,1], W^{k,p})}).\end{aligned} \end{equation}By the Sobolev embedding, for $\frac 1 {q^*}=\frac1 q- \frac 1 n = \frac 1p+ \frac 1p-\frac{ 2}n$, 
\[ \Vert \rho \Vert _{L^\ell ([\delta ,1], W^{k,q^*})} \le C (\Vert \rho \Vert _{L^{2\ell} ([0,1], W^{k,p})}, \ell ,\delta ,p,n,k)<\infty.\]

This shows if $\frac n2 <p< n$, $\rho$ gained higher integrability in $x$. Iterating it finite times, we can make $q$ in \eqref{eq-837} becomes larger than $n$. If $p>n$, we directly reach at $q>n$: since $\rho\in L^\ell_t W_x^{k,n-\eps}$ for all $\eps>0$, $r$ in \eqref{eq-r} can be made arbitrarily large and $p>n$. Thus we can make $\frac1 q = \frac 1 p + \frac 1 r <\frac{1}{n}$.   If $p=n$,  we have to apply iteration one time with $p=n-\eps $ so that it reduces to $p>n$ case. Once  \eqref{eq-837} for $q>n$ is attained, Morrey inequality then gives $W^{k,\infty}$ bound. Applying this again for $k$ replaced by $k+1$, we have the following: for all $1<\ell<\infty$, there exists $\ell'=\ell'(\ell,p,n)\in(\ell,\infty)$ such that 
\begin{equation}\label{eq-1661}  \Vert \rho \Vert_{L^{\ell} ([\delta,1], W^{k+1,\infty })} \le C(\Vert \rho \Vert _{L^{\ell'} ([0,1],W^{k,p})},\ell ,\delta,p, n, k).\end{equation} 
Finally, iterating \eqref{eq-1661} for three times $k$, $k+1$, and $k+2$,
\[ \Vert \rho \Vert_{L^{2}([\delta,1],W^{k+1,p})}+\Vert\partial_t  \rho \Vert_{L^{2}([\delta,1],W^{k+1,p})} \le C(\Vert \rho \Vert _{L^\infty([0,1],W^{k,p}) }, \delta, p,n,k ).\]
This gives desired bound on $\Vert \rho \Vert_{L^\infty ([\delta,1], W^{k+1,p})} $, finishing the one step improvement, and thus  \eqref{eq-smoothinglem1} (and consequently \eqref{eq-smoothinglem2}) follow.
\end{proof}
\begin{lemma}\label{lemma-lowerbd} Under the same assumption of Lemma \ref{lemma-regularityrho}, there exists $C=C(\Vert \rho \Vert_{L^\infty_t L^p_x},V,W)<\infty$ such that 
\[ \rho(x,t) \ge C^{-1}>0,\]
for all $x\in \Tn$ and $t\in [1/2,1]$.\end{lemma}
\begin{proof}Write the equation in the form $\rho_t -\Delta \rho = \nabla \cdot (\rho B)$, where $B=\nabla V+ (\nabla W)* \rho $. Note that the vector field $B$ belong to $L^\infty([1/4,1]\times \Tn)$ due to the regularity of $\rho$ obtained by Lemma \ref{lemma-regularityrho}.
By Moser-type parabolic Harnack inequality for nonnegative solutions, 
\[\inf_{t\in[\frac{1}{2},1],\,  x\in \Tn }\rho(x,t) \ge c_0 \sup_{t\in[ \frac{1}{4},1], x\in \Tn  } \rho(x,t)\]
for some $c_0=c_0(\Vert B \Vert _{L^\infty([1/4,1]\times \Tn )},n)>0. $ 
\end{proof}

\begin{remark}\label{remark-regularityrho}
  Lemma \ref{lemma-regularityrho} and Lemma \ref{lemma-lowerbd} also show that any $\rho \in L^{p}$, $p>\max(\frac{n}{2},1)$, which is a solution  to the elliptic equation 
    \[0=\Delta \rho +\nabla \cdot (\rho \nabla V+ \rho \nabla (W*\rho))\] in the sense of distribution is indeed smooth and strictly positive. 
\end{remark}

\begin{lemma} \label{lemma-subsequentiallimit}  Let $\rho(t)$ be a solution to \eqref{MV} for $t>0$. For $p>\max(\frac{n}{2},1)$, suppose $\Vert \rho(t)\Vert _{L^p}\le L<\infty $ for $t\ge1$. For every $\tau_i\to \infty$, there exists a further subsequence (still denoted by $\tau_i$ here) and a stationary (time-independent) smooth solution $\rho_*\in C^\infty$ such that 
\[\rho(\tau_i+\cdot ) \to \rho_* \text{ in } W^{k,p}([-\ell,\ell]\times \Tn) \text{ for all }\ell <\infty  \text{ and } k\in \mathbb{N}  .\]
\begin{proof} From the bounds on $\Vert V\Vert_{L^{\infty}}$ and $\Vert |z|^{n-2} W(z)\Vert_{L^\infty}$ and the fact that $\Tn$ is compact, we have a  lower bound on $\mathcal{F}$ as  $\mathcal{F}(\rho dx) \ge -C$ for some $C=C(\Vert \rho \Vert_{L^\infty(\Tn)})<\infty $. By Lemma~\ref{lemma-regularityrho}, we have a uniform bound on $\Vert \rho(t)\Vert_{L^\infty}$ for all $t\ge2$. Thus $\mathcal{F}(\rho(t)dx)$ monotone decreases to $\alpha_\infty:=\inf_{t} \mathcal{F}(\rho(t)dx)> -\infty $. Given a sequence $\tau_i\to \infty$, define $\rho_i(x,t):= \rho(x,\tau_i+t)$. By Lemma~\ref{lemma-regularityrho}, we can pass to a subsequence and obtain that $\rho_i(x,t)\to \rho_*(x,t)$ smoothly on every compact sets $\Tn\times [-\ell,\ell]$.  Smooth convergence of $\rho_i$ and the lower bound on density from Lemma \ref{lemma-lowerbd} imply that $\mathcal{F}(\rho_* (t)dx) =\lim_{i\to \infty} \mathcal{F}(\rho_{i}(t)dx) = \alpha _\infty$ for all $t\in \mathbb{R}$. In view of the energy dissipation identity 
\[\frac{d}{dt} \mathcal{F}(\rho_*(t)dx) = \int - |\nabla (\log \rho_* +V+ W*\rho_*)|^2 \rho_* dx,  \]$\rho_*$ is a smooth stationary solution to $\nabla \cdot (\rho (\nabla(\log \rho +V+ W*\rho))=0$.

\end{proof}

\end{lemma}

%

\begin{lemma} \label{lemma-d2L1} Let $(\Sigma,g)$ be a compact Riemannian manifold. For two probability measure $\mu_1$ and $\mu_2$ on $\Sigma$,
\[d_2(\mu_1 ,\mu_2 )^2  \le \mathrm{diam}(\Sigma)^2\, |\mu_1-\mu_2|(\Sigma).\]
\begin{proof}
    Note that $d_2(\mu_1 ,\mu_2 )^2 \le  \mathrm{diam}(\Sigma) d_1(\mu_1 ,\mu_2 ).$ Here $d_p$ is the $p$-Wasserstein distance. By the Kantorovich-Rubinstein duality, 
    \[d_1(\rho_1dx,\rho_2dx)= \sup _{\mathrm{Lip}(\phi)\le 1 } \int \phi(x) d\mu_1(x) - \phi(x) d\mu_2(x) .\]For given $\phi$ with $\mathrm{Lip}\phi \le 1$, choose a point $x_0\in \Sigma$ and we have
\[\int \phi d(\mu_1-\mu_2) =  \int (\phi-\phi(x_0))  d(\mu_1-\mu_2)\le \mathrm{diam}(\Sigma)\int d|\mu_1-\mu_2|. \]
\end{proof}
\end{lemma}

\begin{lemma}\label{lemma-d2interpolation}

Let $\rho_*dx$, $\rho_1dx \in \mathcal{P}^{ac}(\Tn) $ and  assume $\Vert \rho_* \Vert_{W^{2,p}}$, $\Vert \rho_1 \Vert_{W^{2,p}}\le C<\infty $. For given $\varepsilon>0$ there exists $\delta=\delta(\varepsilon,C)>0$, such that 
\[ d_2(\rho_*,\rho_1)\le \delta \text{\quad  implies\quad  } \Vert \rho_* -\rho_1\Vert_{W^{1,p}} \le \eps  . \]
   
\begin{proof}
It is straightforward from a compactness argument. Suppose the assertion false and thus there exist sequences $\rho_{0,n}$ and $\rho_{1,n}$ such that $\Vert \rho_{i,n}\Vert_{W^{2,p}}\le C$ for all $i=1,2$, $n=1,2,\ldots$, 
\begin{equation}\label{eq-tt1}d_2(\rho_{0,n},\rho_{1,n}) \le 1/n  ,\end{equation}and $\Vert \rho_{0,n}-\rho_{1,n}\Vert_{W^{1,p}} \ge \varepsilon_0 >0$ for some $\varepsilon_0>0$. 

By Rellich–Kondrachov theorem, by passing to a subsequence, we can assume $\rho_{0,n}\to \rho_*$ and $\rho_{1,n}\to \rho_1$ both (strongly) in $W^{1,p}$. In view of \eqref{eq-tt1} and Lemma \ref{lemma-d2L1}, $d_2(\rho_*,\rho_1)\le d_2(\rho_*,\rho_{0,n})+d_2(\rho_{0,n},\rho_{1,n})+d_2(\rho_{1,n},\rho_1)$, where the right hand side converges to $0$ as $n\to \infty$. i.e., $\rho_*=\rho_1$. This shows $\Vert\rho_{0,n}-\rho_{1,n}\Vert_{W^{1,p}}\to 0$, a contradiction. 

\end{proof}
\end{lemma}

\begin{proof} [Proof of Theorem \ref{thm-mainconvergence}]Choose $\tau_i \to \infty$ and let $\rho_*$ be the smooth stationary solution attained as a limit of $\rho (\tau_i +\cdot)$ in Lemma \ref{lemma-subsequentiallimit}. Our aim is to prove $\rho(t)$ indeed converges to $\rho_*$.

 In view of Lemma \ref{lemma-regularityrho} and Lemma \ref{lemma-d2interpolation}, the \L ojasiewicz inequality shown in Theorem \ref{thm-main-LS} reduces to the following statement: there exists $\sigma_0>0$ such that $d_2(\rho(t)dx,\rho_*dx) \le \sigma_0$ at $t\ge2$, then 
\be \label{eq-LSt}  \Vert (\nabla (\log \rho(t)   +  V +  W *\rho(t))\Vert_{L^2_{\rho(t)}}  \ge c\, (\mathcal{F}(\rho(t))-\mathcal{F}(\rho_*))^\theta .\ee Note we also used $|\mathcal{F}(\rho(t))-\mathcal{F}(\rho_*)|=(\mathcal{F}(\rho(t))-\mathcal{F}(\rho_*))$.
For a sake of convenience, let us define 
\[Y(t):= -\nabla ( \log \rho(t) +  V +  W*\rho(t)) \]
Writing the equation of $\rho$ in the form of continuity equation, we have  
\[ \partial_t \rho+ \nabla \cdot (\rho  Y)=0.\]

\medskip

Suppose we have a time interval $[t_1,t_2]$ with $t_1\ge 2$ and $t_2<\infty$ such that $d_2(\rho(t)dx,\rho_*dx)<\sigma_0$ for $t\in[t_1,t_2]$. For such $t\in [t_1,t_2]$,
\[\frac{d}{dt} \mathcal{F}(\rho)=- \int |Y|^2 \rho dx   \le -c(\mathcal{F}(\rho)-\mathcal{F}(\rho_*))^{\theta}\Big( \int |Y|^2 \rho dx \Big)^{1/2}.\] 
Integrating in time from $t_1$ to $t_2$ using Benamou-Brenier formula \cite{BenamouBrenier2000}, 
\begin{equation}\label{eq-1222} d_2(\rho(t_2),\rho(t_1))\le \frac{1}{c(1-\theta)} (\mathcal{F}(\rho(t_1))-\mathcal{F}(\rho_*))^{1-\theta} .\end{equation} 

\medskip

Now let us choose $\tau>>1$ such that \begin{equation}\label{eq-1111} d_2(\rho(\tau),\rho_*)+\frac{1}{c(1-\theta)} (\mathcal{F}(\rho(\tau ))-\mathcal{F}(\rho_*))^{1-\theta} \le \frac{\sigma_0}2 .\end{equation} This is possible since $d_2(\rho(\tau_i)dx,\rho_* dx)\to 0$ by Lemma \ref{lemma-d2L1} and $\mathcal{F}(\rho(\tau_i)dx)\to \mathcal{F}(\rho_*dx)$ as $i\to\infty$.   \eqref{eq-1222} and \eqref{eq-1111} imply that $d_2(\rho(t)dx,\rho_*dx) <\sigma_0$ for all $t\ge \tau$. If not, there exists the first time $\tau'$ such that $d_2(\rho(\tau')dx,\rho_*dx) =\sigma_0$. Plugging $t_1=\tau$ and $t_2 =\tau'$, \eqref{eq-1222} and \eqref{eq-1111} imply $d_2(\rho(\tau_i), \rho(\tau')) \le \sigma_0/2$, a contradiction. As a consequence, for all $t_1\ge \tau$, letting $t_2=\tau_i$ and taking $i\to \infty$ in \eqref{eq-1222}, 
\begin{align}
d_2(\rho_*,\rho(t_1)) \le \frac{1}{c(1-\theta)}|\mathcal{F}(\rho(t_1))-\mathcal{F}(\rho_*)|^{1-\theta}. \label{FLS}
\end{align} 
This implies the convergence in the Wasserstein distance $d_2(\rho(t),\rho_*)\to 0$ as $t\to \infty$. The smooth convergence follows by Lemma \ref{lemma-regularityrho} and the fact that a smooth convergence implies  the convergence in $d_2$. 

\end{proof}

We obtain a rate of convergence, Corollary~\ref{corollary-rate}, as a direct conseuqnece of \L ojasiewicz inequality. 
\begin{proof}[Proof of Corollary~\ref{corollary-rate}] 
Let $z(t) := \cF(\rho(t)) - \cF(\rho_*)\ge 0 $. By the energy dissipation identity,
\[
z'(t)
= \frac{d}{dt}\cF(\rho(t))
= - \int_{\T^n} |Y(t,x)|^2 \,\rho(t,x)\,dx
= - \| Y(t) \|_{L^2_{\rho(t)} }^2,
\]
where $
Y(t,x) := \nabla\big( \log \rho(t,x) + V(x) + W * \rho(t,x) \big).$ 

We may choose $\tau\ge2$ as in the proof of Theorem \ref{thm-mainconvergence} so that
$\|\rho(t)-\rho_*\|_{W^{1,p}} \le \sigma$ for all $t\ge\tau$.
Hence, for all $t\ge\tau$, $\| Y(t) \|_{L^2_{\rho(t)}} \ge c\, z(t)^\theta,$ and therefore
\begin{equation}\label{eq:z-ODE}
  z'(t)
  = - \| Y(t) \|_{L^2_{\rho(t)}}^2
  \le - c^2 z(t)^{2\theta},
  \quad t\ge\tau.
\end{equation}
Assume $t \geq \tau$ from now on. As $z'(t)=0$ at some large $t$ implies $\rho(t)$ is stationary (and $\rho(t)=\rho_*$), we may assume $z(t)>0$ for all $t>0$.

\medskip

If $\theta = \frac12$, from \eqref{eq:z-ODE} we get $z'(t) \le -c^2 z(t)$. By Grönwall's inequality,
\[
z(t) \le z(\tau) e^{-c^2 (t-\tau)} \le C e^{-c^2 t}.
\]
If $\theta \in \left(\frac12,1\right)$,  using \eqref{eq:z-ODE} we compute
\[
\frac{d}{dt} z(t)^{1-2\theta}
\ge (1-2\theta)\,z(t)^{-2\theta} \big(-c^2 z(t)^{2\theta}\big)
= (2\theta-1)c^2 > 0.
\]
Integrating from $\tau$ to $t$ yields
\[
z(t)^{1-2\theta}
\;\ge\; z(\tau)^{1-2\theta} + (2\theta-1)c^2 (t-\tau).
\]
Since $1-2\theta < 0$, the function $s\mapsto s^{1-2\theta}$ is strictly
decreasing on $(0,\infty)$, and the above implies
\[
z(t)
\le \Big( z(\tau)^{1-2\theta} + (2\theta-1)c^2 (t-\tau) \Big)^{\frac{1}{1-2\theta}}
\le C (1 + t)^{-\frac{1}{2\theta-1}}.
\]Combining \eqref{FLS} with the bounds on $z(t)$, we conclude \eqref{rate}.
\end{proof}
We expect that finer asymptotics for the difference $\rho (t)-\rho_*$ (including sharp convergence rates with matching lower bounds) could be obtained by extending the recent result of the first named author and Hung \cite{choi2024thom} to the Wasserstein gradient flow setting. We plan to address this in future research.

\section{Application}

We demonstrate that sufficiently regular $W$ rules out blow-up and ensures convergence. This type of argument and result would to extend to different settings. We present the following as a representative example.
\begin{theorem} \label{thm-noblowup} In addition to \textup{\textbf{(\ref{assump_A1})}--\textbf{(\ref{assump_A3})}}, if we further assume $W\in W^{1,\infty}(\Tn)$, then for all $\delta>0$, $\sup_{t\ge\delta}\Vert \rho(t)\Vert_{L^\infty}\le C$ for some uniform $C=C(n,V,W,\delta)$. Thus $\rho(t) $ always converges to a limit stationary solution $\rho_*$, which is selected by the initial data. 
\begin{proof} Write the equation in the form
\begin{equation}\label{eq-566} \rho_t -\Delta \rho = \nabla \cdot (\rho B),
\end{equation}
where $B=\nabla V+ (\nabla W)* \rho $. Since $\nabla W\in L^\infty$ and $\rho(t)$ is a probability density, $\Vert B \Vert _{L^\infty}\le C(\Vert \nabla V \Vert_{L^\infty} +\Vert \nabla W\Vert_{L^\infty} )$. We claim that $\Vert \rho(t)\Vert_{L^2}\le C(V,W)$ for $t\ge \delta/2$. This, together with De Giorgi–Nash–Moser theorem on $L^2 \to L^\infty$ estimate for parabolic equation \eqref{eq-566} implies that 
\[\Vert \rho(t)\Vert_{L^\infty}<C(V,W), \text{ for }t\ge \delta.\]
Now the result follows from Theorem \ref{thm-mainconvergence}.

It remains to prove the claim. By the energy estimate on \eqref{eq-566}, i.e., multiply $\rho$ and integrate in $x\in \Tn$,
\[ \partial_t \Vert \rho(t) \Vert_{L^2}^2 \le -\Vert \nabla \rho(t) \Vert_{L^2}^2 +C\Vert \rho(t)\Vert_{L^2}^2, \] for $C=C(V,W)$. In view of the Gagliardo–Nirenberg inequality \cite{nirenberg1959elliptic}
$\Vert \rho  \Vert_{L^2}^{2+\frac{4}{n}} \le C(\Vert \nabla \rho  \Vert_{L^2}^2 \Vert \rho \Vert_{L^1}^{\frac 4n} +\Vert \rho \Vert _{L^1}^{2+\frac{4}{n}})$, 
\[\partial_t \Vert \rho(t)\Vert_{L^2}^2 \le -\frac{1}{C}\Vert \rho(t)\Vert_{L^2}^{2(1+\frac{2}{n})}+C(\Vert \rho(t)\Vert_{L^2}^2 + 1).\]
Integrating this differential inequality $
\Vert \rho(t)\Vert_{L^2}^2 \le C(t^{-\frac{n}{2}}+1)$, for some $C=C(V,W)$. This proves the claim.
\end{proof}
\end{theorem}
In the previous proof, a key ingredient is the local-in-time $L^2$ estimate that is independent of the initial data. One may also use this estimate to establish the existence of solutions for arbitrary $L^1$ initial data.

\medskip
Next, we present an application to the Keller-Segel equation, which serves as a canonical example where the Coulomb-type singularity and analyticity assumptions on $W$ (\textup{\textbf{(\ref{assump_A2})}--\textbf{(\ref{assump_A3})}}) are naturally satisfied.




\begin{theorem}[Keller-Segel chemotaxis model] \label{cor-KSeq}Consider the parabolic-elliptic Keller-Segel models \eqref{eq-KS} possibly with an additional flux induced by the potential $V$. If $V$ is analytic and a given solution does not blow up, then the solution converges smoothly to a stationary solution as $t\to \infty$. 

\begin{proof} As discussed in Example \ref{example_KS}, we may write the model with McKean-Vlasov equation with $W= \chi G_\beta$, for $\beta=0$ or $\beta>0$. It suffices to check that $G_\beta$ satisfies two conditions of $W$. Without loss of generality, we assume $\chi=1$. First, \textbf{(\ref{assump_A2})} follows immediately from the uniqueness of green function $G_{\beta}$ (with mean-zero condition when $\beta=0$) and the fact $z\mapsto G_\beta (-z)$ is also a Green function.  For a convenience, let us view $G_\beta$ as a $\mathbb{Z}^n$-periodic function on $\mathbb{R}^n$.

To prove \textbf{(\ref{assump_A3})} for $n\ge3$, it suffices to prove bound $G_{\beta}(z)\le C_{\beta} |z|^{2-n}$ on a neighborhood of $0\in \Tn$. We recall the following interior estimates for uniformly elliptic equations with analytic (here constant) coefficients: let $u$ be a solution to the equation either 
\[ \Delta u=0 \quad  \text{ or } \quad \Delta u -\alpha u =0, \]
for some $\alpha>0$ on $B_r\subset \mathbb{R}^n$. Then there holds 
\[ |\nabla^k u(0) |\le C_1 \Big ( \frac{C_2}{r}\Big )^{k} \, k! \, \sup_{B_r} |u|.\]
Here $C_1$ and $C_2$ are constants depend only on $n$ and $\alpha>0$. In view of this estimate on the analyticity of solutions, for \textbf{(\ref{assump_A3})} it suffices to obtain  bound on the growth $G_{\beta}$ in the following form:   $|G_\beta (x)| \le C |x|^{2-n} $ on a small neighborhood $B_r\setminus \{0\}$. For $\beta>0$, this implies 
\[|\nabla^k G_\beta (x)|\le  C_1 \frac{C_2^{k}}{|x|^k}\, k! \, \sup_{B_{\frac{|x|}{4}}(x)} |G_\beta| \le C' \frac{C_2 ^k}{|x|^{n-2+k}} \, k!.    \]
In the case $\beta=0$, we argue similarly for $G_0(x)+\frac{|x|^2}{2n}$ obtain desired estimate.  With this purpose, on $\mathbb{R}^n$ consider the Yukawa potential (Newtonian potential for $\beta=0$), $N_\beta:\mathbb{R}^n \to \mathbb{R}$ defined by 
\begin{equation}\label{eq-Nb}N_{\beta }(x):=-\int_0^\infty (4\pi t)^{-\frac n2} e^{-\beta t} e^{-\frac{|x|^2}{4t }} dt . \end{equation} Indeed they are green functions: $N_{0}(z)$ is Newtonian potential which satisfies $\Delta N_0 =\delta_0$ and, for $\beta>0$, $\Delta N_{\beta }-\beta  N_{\beta}=\delta_0 $. Since $e^{-\beta t } \le1$ for all $\beta\ge0$, we estimate 
\[| N_{\beta}(x) |\le| N_0(x)| = c_n |x|^{2-n} ,  \]  for $|x|\le 1$. If $\beta=0$, $u_0:=(G_0+ \frac{|x|^2}{2n}) - N_0$ is harmonic on a neighborhood of $0$. If $\beta>0$, $u_{\beta}:=G_\beta -N_\beta $ solves $\Delta u_\beta -\beta u=0$. By the maximum principle, $u_\beta$ is bounded on a neighborhood of $0$. This shows desired bound on the growth of $G_\beta$ near $0$. 

In the case $n=1$ or $2$, we consider instead the first derivative\footnote{When $1\le n\le 2$ and $\beta=0$, the heat-kernel representation \eqref{eq-Nb} diverges as $t\to\infty$. Instead,
one may define the Newtonian potential via a renormalization:
\[
\mathcal N_0(x):= \lim_{\tau\to\infty} -\int_0^{\tau}(4\pi t)^{-n/2}e^{-|x|^2/4t}\,dt+\ell(\tau),
\]
where $\ell(\tau)$ is chosen according to the desired normalization.
Since we only need estimates for derivatives, we work with $\nabla \mathcal N_0$, for which differentiating under the integral sign yields a convergent integral even when $n=1,2$.} of $G_\beta$. Observe $\partial_{x_i}(G_0+\frac{|x|^2}{2n}- N_0)$ and $\partial_{x_i} (G_\beta - N_{\beta})$ satisfy $\Delta u=0$ and $\Delta u -\beta u =0$, respectively. Moreover
\[\partial_{x_i} N_{\beta} (x)= -\int_0^\infty -\frac{x_i}{2t} \frac{1}{(4\pi t)^{\frac n2}} e^{-\beta t} e^{-\frac{|x|^2}{4t}} dt \]
implies $|\nabla N_{\beta}|\le C_\beta |x|^{1-n}$. This proves $|\nabla G_\beta (x)|\le {C_\beta} |x|^{1-n}.$ Now following a similar argument as before,
\[ |\nabla ^{k+1} G_{\beta}(x)| \le C' \frac{C_2 ^k}{|x|^{n-1+k}} k!,  \] for all $k=0,1\ldots$, and this shows desired condition.

\end{proof}
    
\end{theorem}

Our method works for radial interaction though it does not extends to a smooth function on torus.

\begin{theorem}[radial interaction]\label{thm-radialpotential}
Let $W(x-y)=\om(d^2(x,y))$ and assume that $\om$ satisfies
\begin{equation}\label{omega_growth}
 | \om^{(k)}(r) | \leq k! C_\om^{k+1}{r^{\frac{2-n}{2}-k}} , \quad k=1,2,\ldots .
\end{equation}
Here \eqref{omega_growth} is assumed in place of \textbf{(\ref{assump_A3})}.
Then the main Theorems \ref{thm-main-LS} and \ref{thm-mainconvergence} hold for the McKean-Vlasov equation \eqref{MV}. 
\end{theorem}
\begin{proof} 

Condition \eqref{omega_growth} implies that $W(z)=\om (d^2(0,z))$ satisfies \textbf{(\ref{assump_A3})} everywhere except on the cut locus of $0$, denoted by $\mathcal{C}_0 := \{z\in \Tn : z_i=1/2 \text{ mod } 1 \text{ for some } i\}$. The singularity on $\mathcal{C}_0$ is negligible because $\mathcal{C}_0$ has measure zero (codimension 1), affecting neither the integral estimates nor the dominated convergence arguments. We sketch the necessary modifications below.

\medskip

For $\Phi \in \mathring{W}^{2,p}$, $p>n$, assume $\Vert \Phi \Vert_{\mathring{W}^{2,p}}\le \sigma_0 $. If $\sigma_0$ is small, then $x\mapsto x+\Phi(x)$ is a $C^{1,\beta}$-diffeomorphism of $\Tn$. Define the collection of all pairs of cut loci
\[\mathrm{Cut}:= \{(x,y)\in (\Tn)^2  : x_i-y_i =\tfrac12  \, (\text{ mod }1)  \text{ for some }i=1,\ldots n\} . \]
Also define
\[\mathrm{Cut}(\Phi ):=\{(x,y) \in (\Tn)^2: ( x+\Phi(x),y+\Phi(y)) \in \mathrm{Cut} \} .\] 
Note that $\nabla^k W (x+\Phi_x-y-\Phi_y)$ is smooth outside of $\mathrm{Cut}(\Phi)\cup \Delta$ for all $k$. Here $\Delta\subset \Tn \times \Tn$ is the diagonal set. Since $\mathrm{Cut}(\Phi )$ is $2(n-1)$-dimensional, there is a bound on the size of the $\eps$-tubular neighborhood
\[ \mathcal{H}^{2n}(N_\eps ( \mathrm{Cut}(\Phi ))) \le C \eps^2  .\]

In Section \ref{sec-analyticity}, we need to justify the (Gateaux) derivative formulas in \eqref{lemma-1stvariation}, \eqref{eq-dMW} and \eqref{eq-GateauxW}. Consider $\xi\in \mathring{W}^{2,p}$ with $\Vert\xi \Vert _{\mathring{W}^{2,p}}\le1$ and note that in this case $\Vert t \xi \Vert _{C^0} \le  C' |t| $. Consider the difference quotient  
\begin{equation}\label{eq-782}\frac{\tcW(\Phi+t\xi)-\tcW(\Phi)}{t} = \iint D^t _{\Phi,\xi}W(x,y) \, d\rho_*(x)d\rho_*(y),\end{equation}
where $D^t_{\Phi,\xi}W(x,y)$ is equal to 
\[\frac{ W(x+\Phi(x)+t\xi(x)-y-\Phi(y)-t\xi(y)) - W(x+\Phi(x)-y-\Phi(y))}{t}.\]
We divide the region of integration in \eqref{eq-782} into $N_{6C'|t|}(\mathrm{Cut}(\Phi))$ and its complement. On $N_{6C'|t|}(\mathrm{Cut}(\Phi))$, the integrand $D^t_{\Phi,\xi}W(x,y)$ is uniformly bounded by $O(1)|t|^{-1}$ for small $|t|>0$, and since the volume of the tubular neighborhood is bounded by $O(1)|t|^2$, the integral converges to $0$ as $t\to0$. On the complement of the tubular neighborhood, the original argument in Section \ref{sec-analyticity} applies exactly: there exists a geodesic from $x+\Phi_x-y-\Phi_y$ to $x+\Phi_x+t\xi_x-y-\Phi_y-t\xi_y$ on $\Tn$ which never touches the cut locus of $0$, and we use the bounds on the derivative of $\nabla W$ to obtain the desired estimate. As $t\to 0$, the complement of the tubular neighborhood exhausts $\Tn\times \Tn \setminus  \mathrm{Cut}$, and the dominated convergence theorem gives \eqref{lemma-1stvariation}. Indeed, the same argument works for the (Gateaux) derivatives of $\mathcal{M}^{\tcW}$, giving us \eqref{eq-dMW} and \eqref{eq-GateauxW}, and the analyticity of $\mathcal{M}^{\tcW}$ follows.

\medskip 

In Section \ref{sec-fredholm}, there is an estimate involving the integral of $\nabla^2 W(x_1-x_2)$ in Lemma \ref{HLS}. Since the cut locus of $x_1$ is $n-1$-dimensional for all $x_1$, the proof and the estimate do not change at all. 

\medskip

In Section \ref{sec-LS}, we need a pointwise estimate of 
\begin{align*}
\delta _{\Phi}K(x,y) := \na^2 W(x-y + \Phi(x)-\Phi(y)) - \na^2 W(x-y),
\end{align*}
namely $|\delta _{\Phi}K(x,y)|\le C \Vert \Phi\Vert _{\mathring{W}^{2,p}} d(x,y)^{-n}$, in the proof of Lemma \ref{lemma-M-L}. In this estimate, we consider a geodesic joining $x+\Phi(x)-y-\Phi(y)$ and $x-y$ in $\Tn$ and use \eqref{eq-KM3}. We claim that \eqref{eq-KM3} still holds even in the case when the geodesic crosses the cut locus of $0$. By a direct computation, if $z$ is not at the cut locus of $0$, then $\nabla^2 W(z) = 4 \omega''(d^2 (0,z)) z_T\otimes z_T  +2w'(d^2(0,z))I_n $. Here $z_T$ is the unique lift of $z$ in $(-1/2,1/2)^n$. Observe that $z\mapsto z_T\otimes z_T$ extends to a continuous function on $\mathbb{R}^n$. This shows that $\nabla^2 W(z)$ is $C^0(\Tn)$ and piecewise smooth (away from the origin). In particular, \eqref{eq-KM3} remains true and the same result follows. 

\medskip

Finally, in Section \ref{sec-proofmainthm}, we used the fact that $\nabla W(z)$ is an $L^1$ function in the proof of Lemma \ref{lemma-regularityrho}, and it remains true in this case as well.

\end{proof} 
Let us remark on some implications of our method for non-analytic potentials. The analyticity becomes crucial when the target stationary solution is degenerate (i.e., $\ker L$ is non-trivial): indeed, in Corollary \ref{cor-LSnondegenerate}, the \L ojasiewicz inequality was established without analyticity provided that the stationary solution is non-degenerate. Let $\rho(t)$ be a given solution to \eqref{MV} that does not blow up. Following Section \ref{subsec-convergence}, one obtains that $\rho(t)$ has an $\omega$-limit. If one such $\omega$-limit, say $\rho_*$, is non-degenerate, Corollary \ref{cor-LSnondegenerate} implies that $\rho(t)$  converges to $\rho_*$ exponentially. A simple sufficient condition guaranteeing the non-degeneracy of stationary solutions is the (uniform) convexity of $V$ and $W$. Using this approach one may recover or slightly improve some classical results which are based on the displacement convexity or Bakry-Emery method.



\subsection*{Acknowledgments}
 The authors are were supported by  the National Research Foundation(NRF) grant funded by the Korea government (MSIT) (RS-2023-00219980). BC and SJ have been supported by Samsung Science \& Technology Foundation grant SSTF-BA2302-02.

 We would like to thank Pei-Ken Hung, Young-Heon Kim, Soumik Pal, D\'aniel Virosztek, Cale Rankin, Antoine Diez, Pierre Monmarch\'e, and Matthew Rosenzweig for their interest and insightful comments.

\bibliography{LS.bib}
\bibliographystyle{alpha-short}
\end{document}